\documentclass[a4paper,draft,12pt]{amsproc}
\usepackage{amssymb}
\usepackage[hyphens]{url} 
\usepackage{enumerate}
\usepackage[shortlabels]{enumitem}
\usepackage[T1]{fontenc}
\usepackage[latin1]{inputenc}
\usepackage{relsize}
\usepackage{floatrow}
\usepackage{amsthm}
\usepackage{tikz}
\usepackage{pgfplots} 
\usepackage{mathtools}
\usepackage{amsfonts}
\usepackage{amsmath}
\usepackage{faktor}
\usepackage{pict2e}
\usepackage{picture}
\usepackage{lmodern}
\usepackage{graphicx}
\usepackage{layout}
\usepackage{relsize}
\usepackage{calc}
\usepackage{comment} 
\usepackage{titlesec}
\usepackage[english]{babel}
\usepackage{bigints}
\titlelabel{\thetitle.\quad}
\newlength{\depthofsumsign}
\newlength{\totalheightofsumsign}
\newlength{\heightanddepthofargument}

\titleformat{\section}
{
\center\normalsize}
{\thesection.}{1em}{}

\titleformat{\subsection}
{\normalsize\bfseries}
{\thesubsection.}{1em}{}

\makeatletter
\newcommand*{\DivideLengths}[2]{%
  \strip@pt\dimexpr\number\numexpr\number\dimexpr#1\relax*65536/\number\dimexpr#2\relax\relax sp\relax
}
\makeatother

\newcommand{\N}{\mathbb{N}}

\newcommand{\R}{\mathbb{R}}
\newcommand{\C}{\mathbb{C}}

\usepackage{dsfont}

\newcommand{\norme}[1]{\left\Vert #1\right\Vert}

\theoremstyle{plain}
\newtheorem{thm}{Theorem}[section]
\newtheorem{prop}[thm]{Proposition}

\newtheorem{lem}[thm]{Lemma}

\newtheorem{cor}[thm]{Corollary}

\theoremstyle{definition}
\newtheorem{df}[thm]{Definition}
\theoremstyle{remark}

\newtheorem{rq1}[thm]{Remark}

\newtheorem{ex}[thm]{Example}

\newtheorem{q}[thm]{Question}

\theoremstyle{break} 

\theoremstyle{nonumberplain}

\title[$\gamma$-BOUNDEDNESS OF $C_0$-SEMIGROUPS]{ \normalsize $\gamma$-BOUNDEDNESS OF $C_0$-SEMIGROUPS AND THEIR $H^{\infty}$-FUNCTIONAL CALCULI}
\subjclass[2010]{47A60, 47D06}

\keywords{$\gamma$-boundedness, $K$-convex space, $H^{\infty}$-calculus, Half-plane type operators}

\author[L. ARNOLD]{LORIS ARNOLD} 

\address{
LABORATOIRE DE MATHÉMATIQUES DE BESANÇON, UMR 6623, CNRS \\ 
UNIVERSITÉ DE FRANCHE-COMTÉ  \\ 
25030 BESANÇON CEDEX\\
FRANCE}
\email{loris.arnold@univ-fcomte.fr}

\thanks{This work is supported by the French 
``Investissements d'Avenir" program, 
project ISITE-BFC (contract ANR-15-IDEX-03).} 

\begin{document}

{\begin{flushleft}\baselineskip9pt\scriptsize

\end{flushleft}}
\vspace{18mm} \setcounter{page}{1} \thispagestyle{empty}

\begin{abstract}

In this article we discuss the notion of $\gamma$-$H^{\infty}$-bounded calculus, 
strong $\gamma$-$m$-$H^{\infty}$-bounded calculus on half-plane and weak-$\gamma$-Gomilko-Shi-Feng 
condition and give a connection between them. Then we state a characterization of generation 
of $\gamma$-bounded $C_0$-semigroup in $K$-convex space, which leads to a version 
of Gearhart-Prüss on $K$-convex space. 

\end{abstract}

\maketitle


\section{INTRODUCTION}

The $H^{\infty}$-functional calculus for a sectorial operator and a strip-type operator have played and important role in the spectral theory and evolutions equations \cite{haa1}.
The $H^{\infty}$-functional calculus for a half-plane type operator is a recent tool studied in \cite{bat-haa}. 
As for sectorial and strip-type operators, it is natural to construct a holomorphic functional calculus for
 a half-plane type operator $A$ via the Dunford formula
\[
f(A) = \frac{1}{2\pi i} \int_{\partial R} f(\lambda)R(\lambda,A)d\lambda.
\]
Here, $R$ is a half-plane and $f$ is a bounded analytic function on $R$ 
with good properties which ensure that $f(A)$ is bounded. This construction 
allows to define a notion of bounded $H^{\infty}$-functional calculus for a half-plane type operator. 
Contrary to bounded $H^{\infty}$-functional calculus for a sectorial
or a strip-type operator, the bounded $H^{\infty}$-functional calculus for a half-plane type operator has no characterization with simple estimates, even on Hilbert space. However, a weaker notions of bounded $H^{\infty}$-functional calculus, 
called strong $m$-bounded functional calculus (Definition \ref{mbounded}) turn out to be equivalent to a 
condition studied independently by Gomilko \cite{gom} and Shi and Feng \cite{shifeng} (Definition \ref{GFS}), 
called GFS condition in this paper. Furthermore, they show that this condition is sufficient for 
the generation of bounded $C_0$-semigroups. Therefore, the set (containing all 
sectorial operators of type $<\frac{\pi}{2}$) of all half-plane type operators 
which have GFS is included  (equal when $X$ is Hilbert) 
in the set of all negative generators of bounded $C_0$-semigroup. 

The aim of this paper is to study the generation of a $\gamma$-bounded $C_0$-semigroups on a Banach space $X$, 
that is, a $C_0$-semigroup $(T_t)_{t \geqslant 0}$ on a Banach space $X$ such that the set $\{T(t)
\, :\, t\geq 0\}$ is $\gamma$-bounded. A first step is to consider a stronger condition than 
the GFS condition (however equivalent on Hilbert spaces), that we call 
weak $\gamma$-Gomilko-Shi-Feng (Definition \ref{WgammaGFS}) and abreviate as the $W\gamma$-GFS condition. 
It turns out that the $W\gamma$-GFS condition
is equivalent to a notion of $\gamma$-bounded strong $m$-bounded functional calculus.  
Furthermore the latter 
condition is sufficient for the generation of weak $\gamma$-bounded $C_0$-semigroups, and hence
for the generation of $\gamma$-bounded $C_0$-semigroups when the underlying space is $K$-convex. 
Therefore when $X$ is $K$-convex, the set (containing all $\gamma$-sectorial operators of $\gamma$-type $<\frac{\pi}{2}$) of all operators which satisfy the $W\gamma$-GFS condition is equal to the set of 
all negative generators of $\gamma$-bounded $C_0$-semigroups. 
This last statement is contained in our main result
(Theorem \ref{shifengRthm}), which gives equivalence of the
generation of $\gamma$-bounded $C_0$-semigroup with the $\gamma$-m-bounded functional calculus,
as well as with some estimates of the resolvent of the negative generator and its adjoint. 
From this theorem, we deduce, when $X$ is $K$-convex space, 
a $\gamma$-bounded version of the Gearhart-Prüss Theorem (Corollary \ref{corGP}) 
and the following result: if $(T_t)_{t\geq  0}$ is a bounded $C_0$-semigroup on a 
$K$-convex space and if the set $\{e^{-\delta t}T_t\, :\, t\geq 0 \}$ 
is $\gamma$-bounded for one $\delta >0$, then it is $\gamma$-bounded for each $\delta >0$. 
Moreover it is possible to find an example of bounded $C_0$-semigroup 
such that $\{e^{-\delta t}T_t\, :\, t\geq 0 \}$ is $\gamma$-bounded for all $\delta >0$ but not 
for $\delta = 0$.  

Now we describe the structure of the paper. Section 2 only contains preliminary results. 
We recall important results of \cite{bat-haa}, then we collect some results about 
$\gamma$-boundedness and weak $\gamma$-boundedness and generalized square functions. 
In Section 3, we discuss the $W\gamma$-GFS condition and strong $\gamma$-$m$-bounded $H^{\infty}$-functional calculus. We give our main result in Section 4, namely a version of Gomilko-Shi-Feng Theorem on $K$-convex spaces. Section 5 is devoted 
to some consequences of the results of Section 4, in particular
we state a version of the Gearhart-Prüss Theorem in $K$-convex spaces. 
Finally, Section 6 is dedicated to an overview of the implications between the different notions of $H^{\infty}$-bounded functional calculus for half-plane type operator and the generation of bounded and $\gamma$-bounded $C_0$-semigroups.



\section{BACKGROUND AND PRELIMINARY RESULTS}

For any Banach spaces $X,Y$, we let $\mathcal{L}(X,Y)$ denote the space of
all bounded linear operators from $X$ into $Y$. If $Y=X$, we write 
$\mathcal{L}(X)$ instead of $\mathcal{L}(X,X)$. 
If $A$ is a closed operator on $X$, we denote by $Dom(A)$, $\rho(A)$ and $\sigma(A)$ the domain, 
the resolvent set and the spectrum of $A$,
respectively. When $\lambda \in \rho(A)$, we let $R(\lambda,A) = (\lambda I - A)^{-1}$
denote the corresponding resolvent operator.

\subsection{Half-plane type operator.}
In this subsection we review the definitions of half-plane type operators 
and their functional calculi, 
following \cite{bat-haa}. 

We fix a real number $\omega$. For any $\alpha\leq \omega$, we consider  
the right open half-plane 
$$
R_\alpha= \{ z \in \C\, :\, Re(z) > \alpha \}.
$$

\begin{df}
Let $A$ be a closed and densely defined operator on $X$. We will say that $A$ 
is of \textit{half-plane type} $\omega$  if $\sigma(A) \subset \overline{R_{\omega}}$ and 
\[
\forall\, \alpha<\omega,\qquad
\sup \{\norme{R(z,A)}\, :\, Re(z) \leq \alpha \} < \infty.
\]
\end{df}

We will say that a $C_0$-semigroup $(T_t)_{t\geq0}$ is of type $\omega \in \R$ if 
there exists a constant  $C > 0$ such that $\norme{T_t} \leq Ce^{-\omega t}$ for all $t \geq 0$. 
Note that such $\omega$ always exists \cite[Theorem 2.2 section 1.2]{pazy}.  
It follows from the Laplace formula (see e.g. \cite[Formula (7.1) section 1.7]{pazy})
that if $-A$ generates a $C_0$-semigroup $(T_t)_{t\geq0}$ of type $\omega$, then $A$ is an
operator of half-plane type $\omega$ 
(in fact $A$ is even strong half-plane in the sense of \cite[Definition 2.1]{bat-haa}).

Throughout the rest of this subsection, we let $A$ be an operator of half-plane type $\omega$.

For any  $\alpha\leq \omega$, let 
$H^{\infty}(R_{\alpha})$ be the space of all bounded analytic functions $f :R_{\alpha}\to\C$,
equipped with the norm $\norme{f}_{H^{\infty}(R_{\alpha})} := \underset{z \in R_{\alpha} }\sup |f(z)|$.
Then $H^{\infty}(R_{\alpha})$  is a Banach algebra.

Next we consider the auxiliary space 
\[
\mathcal{E}(R_{\alpha}) := \{f \in H^{\infty}(R_{\alpha})
\,:\,\exists s>0, f(z) = O(|z|^{-(1+s)}) \text{ as } |z| \rightarrow \infty\}.
\]

Whenever $\alpha < \delta <\omega$ and $f\in \mathcal{E}(R_{\alpha})$, the integral
\begin{equation}\label{fofA}
f(A) := \frac{1}{2\pi i}\int_{-\infty}^{\infty}f(\delta+it)R(\delta+it,A)dt
\end{equation}
is absolutely convergent in $\mathcal{L}(X)$.
Further its value is independent of $\delta \in (\alpha,\omega)$. This is due
to Cauchy's theorem for vector-valued holomorphic functions.

If $f \in H^{\infty}(R_{\alpha})$, we can define a closed, densely defined, operator 
$f(A)$ by regularisation  as follows (see \cite{bat-haa} and \cite{haa1} for more details).
Let $\mu<\alpha$ and set $e(z) :=(\mu -z)^{-2}$. Then 
$e\in \mathcal{E}(R_{\alpha})$, $ef \in \mathcal{E}(R_{\alpha})$ and 
$e(A) = R(\mu,A)^2$ is injective. Then $f(A)$ is defined by 
$$
f(A) = e(A)^{-1}(ef)(A),
$$
with $Dom(f(A))$ equal to the space of all $x\in X$ such that $[(ef)(A)](x)$ belongs
to the range of $e(A)$ (= $Dom(A^2)$). 
It turns out that this definition does not depend on the choice of $\mu$.

In particular, for any $t\geq 0$, the function $z\mapsto e^{-tz}$ belongs to $ H^{\infty}(R_{\alpha})$, 
hence the above construction provides an operator 
\[
e^{-tA} := (e^{-tz})(A).
\]

The following proposition (see \cite[Proposition 2.5]{bat-haa}) gives an expected link 
between the generator of a $C_0$-semigroup and its exponential in the previous sense.

\begin{prop}\label{propexp}
The operator $-A$ is the generator of a $C_0$-semigroup $(T_t)_{t\geq0}$ if 
and only if $e^{-tA}$ is a bounded operator for all $t \in [0,1]$ and 
$\underset{t\in [0,1]}\sup\norme{e^{-tA}} < \infty$. In this case, we have 
$T_t = e^{-tA}$ for all $t\geq 0$.
\end{prop}

\begin{df}
We say that $A$ has a bounded $H^{\infty}$-functional calculus of type $\omega$
if there exists $C>0$ such that for any $\alpha<\omega$ and for any $f \in H^{\infty}(R_{\alpha})$,
$f(A)$ belongs to $\mathcal{L}(X)$ and
\begin{equation}\label{Infini}
\norme{f(A)} \leq C\norme{f}_{H^{\infty}(R_{\omega})}.
\end{equation}
\end{df}

\begin{rq1}\label{Reduction}
The reasoning at the beginning of \cite[Section 5]{bat-haa} and the so-called convergence 
lemma (see \cite[Theorem 3.1]{bat-haa}) show that to prove
that an operator $A$ has a bounded $H^{\infty}$-functional calculus of type $\omega$,
it suffices to prove an estimate (\ref{Infini}) for any 
$f\in {\mathcal E}  (R_{\alpha})$ and any $\alpha<\omega$.
\end{rq1}

As a direct consequence of Proposition \ref{propexp}, 
we see that if  $A$ has a bounded $H^{\infty}$-functional calculus of type $\omega$, 
then $-A$ generates a $C_0$-semigroup of type $\omega$. 
The converse does not hold true, even on Hilbert space (see Section 6).

The following condition, called Gomilko-Shi-Feng condition (GFS), is important to connect
$C_0$-semigroups and functional calculi.

\begin{df}\label{GFS}
Let $m\geq 1$ be an integer. We say that 
$A$ satisfies $(GFS)_{m, \omega}$  if there exists a constant 
$C>0$ such that 
\[
\int_{\R}|\langle R(\alpha +it,A)^{m+1}x,y \rangle 
 | dt \leq \frac{C}{(\omega - \alpha)^m}\norme{x}\norme{y}
\]
for any $x \in X$ and any $y \in X^*$.
\end{df}

Gomilko \cite{gom} and Shi-Feng  \cite{shifeng} have shown 
the following two results: if $A$ has $(GFS)_{1,0}$, then 
$-A$ generates a bounded $C_0$-semigroup $(T_t)_{t\geq0}$;
conversely if $X$ is a Hilbert space, the negative generator
of a bounded  $C_0$-semigroup satisfies $(GFS)_{1,0}$. 
This is now known as the Gomilko-Shi-Feng Theorem.

\begin{rq1}\label{sechaveGFS}
It is easy to check that a sectorial operator of type $<\frac{\pi}{2}$ has $(GFS)_{1, 0}$. 
We refer e.g. to \cite{haa1} for information about sectorial operators. 
We recall that $A$ is a sectorial operator of type $< \frac{\pi}{2}$ if 
and only if $-A$ generates a bounded analytic semigroup and that in this case,
there exists a constant $C>0$ such that
$$
\forall\, \lambda\in\C\setminus\overline{R_0},\quad \norme{\lambda R(\lambda,A)}\leq C.
$$
This implies that for any $\alpha<0$ and any $t \in \R$, $\norme{R(\alpha+it,B)} \leq \frac{C}{|\alpha+it|}$. 
Hence for any $\alpha<0$ and arbitrary $x \in X$ and $y \in X^*$, we have
\begin{align*}
(-\alpha)\int_{\R}|\langle R(\alpha +it,A)^{2}x,y \rangle  | dt &  
\leq  (-\alpha)\int_{\R}\norme{R(\alpha +it,A)}^{2} dt \norme{x}\norme{y} \\
& \leq  \int_{\R} \frac{-C\alpha}{(-\alpha)^2+t^2}dt \norme{x}\norme{y} \\
 & = C\pi \norme{x}\norme{y},	
\end{align*}
which proves the result.	  
\end{rq1}

It is noticed in \cite[Lemma 5.4]{bat-haa} that for any $\alpha<\beta<\omega$
and any $f \in H^{\infty}(R_{\alpha})$, the $m$-th derivative $f^{(m)}$ of $f$
belongs to $H^{\infty}(R_{\beta})$ for any integer $m\geq 1$. Thus it makes sense to
define $f^{(m)}(A)$.

\begin{df}\label{mbounded} 
Let $m\geq 1$ be an integer.
We say that $A$ has a \textit{strong $m$-bounded functional calculus} 
of type $\omega$ if there exists a constant $C >0$ such that for each 
$\alpha < \omega$ and each $f \in H^{\infty}(R_{\alpha})$,
\[
f^{(m)}(A) \in \mathcal{L}(X) \text{    and     }	
\norme{f^{(m)}(A)} \leq \frac{C}{(\omega - \alpha)^m}\norme{f}_{H^{\infty}(R_{\alpha}) }. 
\]
\end{df}

The following remarkable results are proved in \cite[Theorem 6.4]{bat-haa}: $A$ has
$(GFS)_{m,\omega}$ if and only if $A$  has a strong $m$-bounded functional calculus 
of type $\omega$, if and only if $A$  has a strong $1$-bounded functional calculus 
of type $\omega$. (In particular, $(GFS)_{m,\omega}$  does not depend on $m$.)
Further if $A$ has a strong $m$-bounded functional calculus of type $\omega$, then $-A$ 
generates a $C_0$-semigroup of type $\omega$. It is further shown in \cite[Theorem 7.1]{bat-haa}
that if $X$ is a Hilbert space, then conversely, $A$ has a strong $m$-bounded functional calculus of type $\omega$
if $-A$ generates a $C_0$-semigroup of type $\omega$. This converse is wrong in general, see 
Section 6 for more on this.

The following is implicit in \cite{bat-haa}.

\begin{prop}\label{1implies2}
If $A$ has a bounded $H^{\infty}$-functional calculus of type $\omega$,
then $A$ has a strong $m$-bounded functional calculus of type $\omega$, for any $m\geq 1$. 
\end{prop}

\begin{proof}
According to \cite[Lemma 5.4]{bat-haa}, there is a constant 
$K>0$, such that for any
$\alpha<\omega$ and any $f\in H^{\infty}({R_{\alpha}})$, we have
$\norme{f^{(m)}}_{H^{\infty}(R_{\omega})} \leq  
\frac{K}{(\omega-\alpha)^m}\norme{f}_{H^{\infty}(R_{\alpha})}$.

Assume that $A$ has a bounded $H^{\infty}$-functional calculus of type $\omega$.
Then for any $f$ as above, $f^{(m)}(A)\in{\mathcal L}(X)$ and we have
\[
\norme{f^{(m)}(A)} \leq C\norme{f^{(m)}}_{H^{\infty}(R_{\omega})} 
\leq \frac{KC}{(\omega-\alpha)^m}\norme{f}_{H^{\infty}(R_{\alpha})}. 
\]
This proves that $A$ has a strong $m$-bounded functional calculus of type $\omega$.
\end{proof}

\subsection{The Sun Dual of a $C_0$-Semigroup}\label{Sun}
In this subsection we collect a few facts from \cite{ner}
which are useful when dealing with non reflexive Banach
spaces.

Let $X$ be a Banach space and let $Z\subset X^*$ be a closed subspace. We say that 
$Z$ is norming if there exists a constant $c>0$ such that for any $x\in X$,
$$
c\norme{x}\leq \sup\bigl\{\vert \langle x,y\rangle\vert\,:\, y\in Z,\ \norme{y}\leq 1\bigr\}.
$$

Let $(T_t)_{t\geq 0}$ be a $C_0$-semigroup on $X$, with generator $-A$. It may happen that
the dual semigroup $(T_t^*)_{t\geq 0}$ is not strongly continuous on $X^*$.
We denote by
$X^{\odot}$ (pronounced $X$-sun) the set 
\begin{equation}\label{Xsun}
X^{\odot} := \{x\in X^*, \quad \norme{T^*_t x-x} \underset{t \rightarrow 0}\longrightarrow 0 \}.
\end{equation}
This set  trivially satisfies $T^*_t(X^{\odot}) \subset X^{\odot}$ for every $t \geq0$. 
Moreover $X^{\odot}$ is a closed and weak*-dense
subspace of $X^*$. Indeed we have 
$X^{\odot} = \overline{D(A^*)}$.

We let $T^{\odot}_t$ denote  the restriction  of $T^*_t$ to $X^{\odot}$. Then by definition, 
$(T^{\odot}_t)_{t\geq 0}$ is a $C_0$-semigroup on $X^{\odot}$. 
It is called the \textit{sun dual of $(T_t)_{t\geq 0}$}. Let $A^{\odot}$ be its negative generator. 
Then \textit{$A^{\odot}$ is the part of $A^*$ in $X^{\odot}$} (see \cite[page 6]{haa1} for definition).

We will use the following two results.

\begin{thm}\label{thmsun}
Let $(T_t)_{t\geq 0}$ be a $C_0$-semigroup on $X$, with generator $-A$.
\begin{itemize}
\item [(1)]
We have $\rho(A) = \rho(A^*) = \rho(A^{\odot})$ and $R(\lambda, A)^*y = 
R(\lambda, A^*)y = R(\lambda,A^{\odot})y$ for all $\lambda \in \rho(A) $ and
for any $y \in X^{\odot}$.
\item [(2)]
The space $X^{\odot}$ is norming.
\end{itemize}
\end{thm}

About (2), we note that more precisely, if we define
\[
\norme{x}' := \sup\big\{|\langle x,x^{\odot}\rangle|, \, x^{\odot} \in X^{\odot}, \, \norme{x^{\odot}} \leq 1 \big\},
\]
and if we let $M := \varlimsup_{t \rightarrow 0} \norme{T(t)}$, then we have
\[
\norme{x}' \leq \norme{x} \leq M\norme{x}', \qquad x \in X.
\]

\subsection{$\gamma$-boundedness on Banach spaces.}  
In recent years, $\gamma$-boundedness and $\mathcal{R}$-boundedness played an 
important role in the operator valued harmonic analysis, multiplier theory and functional calculi 
(see \cite{hnvw} for more details). Throughout $X,Y$ denote arbitrary Banach spaces and 
we let $(\gamma_n)_{n \geq 1}$ be a sequence of independent 
complex valued standard Gaussian variables on some probability space $\Sigma$. 
We denote by $G(X)$ the closure of 
\[
\big\{\sum_{k=1}^n \gamma_k \otimes x_k\, :\, x_k \in X, \, n\in \N \big\}
\]
in $L^2(\Sigma,X)$. For $x_1,\ldots, x_n \in X$, we let 
\[
\norme{\sum_{k=1}^n \gamma_k \otimes x_k}_{G(X)} := 
\bigg(\underset{\Sigma}\bigint{\norme{\sum_{k=1}^n \gamma_k(\lambda)x_k}^2d\lambda}\bigg)^{\frac{1}{2}}
\] 
denote the induced norm.

\begin{df}
Let $\mathcal{T} \subset \mathcal{L}(X,Y)$ be a set of operators. We say that
$\mathcal{T}$ is $\gamma$-bounded if
there exists a constant $C \geq 0 $ such that for all finite sequences 
$(T_n)_{n=1}^{N} \subset \mathcal{T}$ and $(x_n)_{n=1}^{N} \subset X$, the following inequality holds:
\begin{equation}\label{Rboundedness}
\norme{\sum_{n=1}^N \gamma_n\otimes T_nx_n}_{G(Y)} \leq C\norme{\sum_{n=1}^N \gamma_n\otimes x_n}_{G(X)}.
\end{equation}
The least admissible constant in the above inequality is called the 
$\gamma$-bound of  $\mathcal{T}$ and we denote this quantity by $\gamma(\mathcal{T})$. 
If $\mathcal{T}$ fails to be $\gamma$-bounded, we set $ \gamma(\mathcal{T}) = \infty$. 
\end{df}

Replacing the sequence $(\gamma_k)_{k\geq 1}$ by a sequence of independent Rademacher 
variables $(\epsilon_k)_{k\geq1}$ in the above definition, we obtain the definition of 
$\mathcal{R}$-boundedness. It is well known that any $\mathcal{R}$-bounded set is $\gamma$-bounded, and that 
these notions are equivalent when $X$ has finite cotype (see \cite[theorem 8.6.4]{hnvw} for a more general result). 
Furthermore if $X$ has cotype $2$ and $Y$ has type 2 (especially when $X=Y$ is an Hilbert space) 
then $\mathcal{R}$-boundedness, $\gamma$-boundedness and uniform boundedness are equivalent.  

Assume that $X$ and $Y$ are Banach lattices with finite cotype. 
By the Khintchine-Maurey inequality \cite [Theorem 7.2.13]{hnvw}, 
there exist $c,C>0$ such that for every $x_1, \ldots, x_n \in X$,
\[
c\norme{\big(\sum_{n=1}^N |x_n|^2\big)^{\frac{1}{2}}}_{X} \leq 	
\norme{\sum_{n=1}^N \gamma_n\otimes x_n}_{G(X)} \leq C
\norme{\big(\sum_{n=1}^N |x_n|^2\big)^{\frac{1}{2}}}_{X},
\]
and $Y$ satisfies a similar property.
Hence a set $\mathcal{T} \subset \mathcal{L}(X,Y)$ is $\gamma$-bounded if 
and only if there exists $C\geq0$ such that for all finite sequences 
$(T_n)_{n=1}^N  \subset \mathcal{T}$ and $(x_n)_{n=1}^N \subset X$,
\begin{equation}\label{Rboundedness2}
\norme{\big(\sum_{n=1}^{N}|T_nx_n|^2\big)^{\frac{1}{2}}}_Y 
\leq C\norme{\big(\sum_{n=1}^{N}|x_n|^2\big)^{\frac{1}{2}}}_X.
\end{equation}

We recall for further use that $\gamma$-boundedness is stable 
under the strong operator topology.  

\begin{prop}\label{R-summarizes}
Let $\mathcal{T} \subset \mathcal{L}(X,Y)$ be  a $\gamma$-bounded set. Then
 the closure $\overline{\mathcal{T}}^{so}$  of $\mathcal{T}$ in the 
strong operator topology is $\gamma$-bounded with 
$\gamma(\overline{\mathcal{T}}^{so}) = \gamma(\mathcal{T}) $.
\end{prop}

We will need the following lemma, for which we refer to \cite[Theorem 8.5.4]{hnvw}.

\begin{lem}[$L^{\infty}$-integral means]\label{Linftymeans}
Let $(\Omega,\mu)$ be a measure space and let $F : \Omega \rightarrow \mathcal{L}(X,Y)$
be an operator-valued function. Assume that $F(\cdot)x$ belongs to $L^1(\Omega;Y)$
for any $x\in X$ and that there exists a constant $K>0$ such that 
$$
\norme{F(\cdot)x}_{L^1(\Omega;X)} \leq K\norme{x},\quad x\in X.
$$
Then, for any $\phi \in L^{\infty}(\Omega)$, we can define a 
bounded operator $T^F_{\phi} \in \mathcal{L}(X,Y)$ by 
\[
T^F_{\phi}x = \int_{\Omega}\phi(s)F(s)xd\mu(s), \quad x\in X,
\]
and for any $C>0$, the set 
\[
\mathcal{T}^F_{\infty} := \big\{T^F_{\phi}\, :\, \phi \in L^{\infty}(\Omega), \, 
\norme{\phi}_{L^{\infty}(\Omega)} \leq C \big \}
\]
is $\gamma$-bounded.
\end{lem}

The adjoint set of a $\gamma$-bounded set may not be $\gamma$-bounded 
(\cite[example 8.4.2]{hnvw}). Following \cite{kal-wei2}, we introduce 
weaker notions to circumvent this difficulty.


\begin{df} \ 
\begin{enumerate}[label = \arabic*)]
\item Let $\mathcal{T} \subset \mathcal{L}(X,Y)$ be a set of operators. 
We say that $\mathcal{T}$ is weak $\gamma$-bounded ($W\gamma$-bounded in short)
if there exists $C$ such that 
for all finite sequences $(T_n)_{n=1}^{N} \subset \mathcal{T}$, 
$(x_n)_{n=1}^{N} \subset X$ and $(y_n^*)_{n=1}^{N} \subset Y^*$, the following inequality holds:
\begin{equation}\label{Wgambounded}
\sum_{n=1}^N |\langle T_nx_n,y_n^* \rangle|\leq C\norme{\sum_{n=1}^N 
\gamma_n\otimes x_n}_{G(X)}\norme{\sum_{n=1}^N \gamma_n\otimes y_n^*}_{G(Y^*)}.
\end{equation}

\item 
Let $\mathcal{T} \subset \mathcal{L}(X,Y^*)$ be a set of operators. We say that $\mathcal{T}$ is 
weak$^*$ $\gamma$-bounded ($W^*\gamma$-bounded in short) 
if there exists $C$ such that for all 
finite sequences $(T_n)_{n=1}^{N} \subset \mathcal{T}$, $(x_n)_{n=1}^{N} 
\subset X$ and $(y_n)_{n=1}^{N} \subset Y$, the following inequality holds:
\begin{equation}\label{W*gambounded}
\sum_{n=1}^N |\langle T_nx_n,y_n \rangle|\leq 
C\norme{\sum_{n=1}^N \gamma_n\otimes x_n}_{G(X)}
\norme{\sum_{n=1}^N \gamma_n\otimes y_n}_{G(Y)}.
\end{equation}
\end{enumerate}
\end{df}

If $\mathcal{T}$ is $W\gamma$-bounded (respectively $W^*\gamma$-bounded) 
then the adjoint set $\mathcal{T}^*$ is $W\gamma$-bounded  
(respectively $W^*\gamma$-bounded) (see \cite[Lemma 2.4]{hof-kal-kuc}). It is 
clear that $\gamma$-boundedness implies $W\gamma$-boundedness, however the converse is false in general. 
Indeed take a $\gamma$-bounded set $\mathcal{T}$ 
such that $\mathcal{T}^*$ is not $\gamma$-bounded, then 
as $\mathcal{T}$ is $W\gamma$-bounded, $\mathcal{T}^*$ 
is also $W\gamma$-bounded but not $\gamma$-bounded.

\begin{lem}[weak $L^1$-integral means]\label{WeakL1means}
Let $(\Omega,\mu)$ be a measure space, let $\mathcal{T}\subset \mathcal{L}(X,Y^*)$ be 
a $W^*\gamma$-bounded set and let $G : \Omega \longrightarrow \mathcal{L}(X,Y^*)$ be
an operator-valued function such that 
$G$ takes values in $\mathcal{T}$ and for all $x \in X$ and $y \in Y$, 
the scalar function $s \longmapsto \langle G(s)x,y \rangle $ is measurable.
 
Then for any $\phi \in L^1(\Omega)$, one can define a bounded operator $T_G^{\phi} \in \mathcal{L}(X,Y^*)$ by 
\[
\langle T_G^{\phi}x,y \rangle := \int_{\Omega}\phi(s) \langle G(s)x,y 
\rangle d\mu(s),\quad \text{for $x \in X$ and $y\in Y$},
\]
and for any $C>0$, the set
\[
\Big\{ T_G^{\phi}\, :\, \phi \in L^1(\Omega),\; 
\norme{\phi}_{L^1(\Omega)} \leq C \Big\} \subset \mathcal{L}(X,Y^{*}) 
\]
is $W^*\gamma$-bounded.		 
\end{lem}

\begin{proof}
Let $\sigma$ be the topology on $\mathcal{L}(X,Y^*)$ generated by the family of seminorms 
$\rho_{x,y}$ defined by $\rho_{x,y}(T) = |\langle T(x),y \rangle |$, for any $x\in X$ and $y\in Y$. 
The topology $\sigma$ is called \textit{weak$^*$ operator topology on} $\mathcal{L}(X,Y^*)$. 
It is clear that if  $\mathcal{T} \subset \mathcal{L}(X,Y^*)$ is $W^*\gamma$-bounded, 
then its closure $\overline{\mathcal{T}}^{\sigma}$ is also $W^*\gamma$-bounded. 
Moreover if $\mathcal{T} \subset \mathcal{L}(X,Y^*)$ is 
$W^*\gamma$-bounded then the absolute convex hull of $\mathcal{T}$ is $W^*\gamma$-bounded
as well. 

With these two facts in hand, 
one can obtain the result by mimicking the proof of \cite[Theorem 8.5.2]{hnvw}. 
Details are left to the reader.
\end{proof}

We will need the notion of $K$-convexity, 
for which we refer to \cite{mau1} or \cite{hnvw}. 
We recall that a Banach space $X$ is 
$K$-convex if and only if there exists a constant
$K > 0$ such that for all $x_1,\ldots,x_N \in X$ the following inequality holds:
\begin{equation}\label{inekconv}
\norme{\sum_{n=1}^N \gamma_n \otimes x_n}_{G(X)} \leq K \sup \Big\{\big|\sum_{n=1}^N \langle x_n,y_n\rangle\big|
\,:\, y_1,\ldots, y_N\in X^*,\ \norme{\sum_{n=1}^N \gamma_n \otimes y_n}_{G(X^*)} \leq 1 \Big\}.
\end{equation}
It turns out that $X$ is $K$-convex if and only if $X^*$ is $K$-convex. If this is the case, then
according to \cite[Corollary 7.4.6]{hnvw}, there exists a constant
$K' > 0$ such that for all $y_1,\ldots,y_n \in X^*$,
\begin{equation}\label{inekconv*}
\norme{\sum_{n=1}^N \gamma_n \otimes y_n}_{G(X^*)} \leq K' \sup 
\Big\{\big|\sum_{n=1}^N \langle x_n,y_n\rangle\big|\, :\,
x_1,\ldots,x_N\in X,\ \norme{\sum_{n=1}^N \gamma_n \otimes x_n}_{G(X)} \leq 1 \Big\}.
\end{equation}
We recall that all UMD spaces are $K$-convex. In particular, $L^p$-spaces are 
$K$-convex for any $1<p < \infty$. Further any closed subspace of a 
$K$-convex space is $K$-convex. 
We also recall that any $K$-convex 
space has a finite cotype. In particular, a $K$-convex Banach space
cannot contain $c_0$. A fundamental result on $K$-convexity is 
Pisier's Theorem \cite[Theorem 7.4.23]{hnvw}
which asserts that $X$ is $K$-convex if and only if $X$ has non-trivial type. 

We note that
there exist non reflexive $K$-convex Banach spaces.
It readily follows from \eqref{inekconv}
that if $Y$ is a $K$-convex space, then a set $\mathcal{T} \subset \mathcal{L}(X,Y)$ is 
$W\gamma$-bounded if and only if it is $\gamma$-bounded. Likewise using \eqref{inekconv*},
we obtain that if $Y$ is $K$-convex, then a set $\mathcal{T} \subset \mathcal{L}(X,Y^*)$ 
is $W^*\gamma$-bounded if and only if it is $\gamma$-bounded.


We now turn to the definition of $\gamma$-spaces, which play a fundamental role in this paper.
Let $H$ be a Hilbert space.
A linear operator $T : H \rightarrow X$ is called $\gamma$-summing if 
\[
\norme{T}_{\gamma} := \sup \norme{\sum_{n=1}^{N}\gamma_n\otimes Th_n }_{G(X)} < \infty,
\]
where the supremum is taken over all finite orthonormal system $\{h_1,...,h_n \}$ in $H$. 
We let $\gamma_{\infty}(H;X)$ denote the space of all $\gamma$-summing operator and 
we endow it with the norm $\norme{\cdot}_{\gamma}$. Then
$\gamma_{\infty}(H;X)$  is a Banach space. Clearly any finite rank (bounded) operator 
is a $\gamma$-summing operator. We let $\gamma(H;X)$ be
the closure in $\gamma_{\infty}(H;X)$ of the space of finite rank operators from $H$ into $X$. 
The spaces $\gamma_{\infty}(H;X)$ and $\gamma(H;X)$ 
do not coincide in general \cite[Example 9.1.21]{hnvw} but when $X$ does not contain a copy 
of $c_0$ (in particular when $X$ is $K$-convex) then these spaces coincide.

Let $(S,\mu)$ be a measure space. We say that a function 
$f : S \rightarrow X$ is weakly $L^2$ if for each $x^* \in X^*$,  
the function $s \mapsto \langle f(s),x^* \rangle $ is measurable and belongs to $L^2(S)$. 
If $f : S \rightarrow X$ is measurable and weakly $L^2$, one can define an opertor $\mathbb{I}_f : L^2(S)\to X$, 
given by
\[
\mathbb{I}_f(g) := \int_S g(s)f(s)ds , \quad g\in L^2(S),
\] 
where this integral is defined in the Pettis sense.

We let $\gamma(S;X)$ be the space of all measurable and weakly $L^2$
functions $f : S \rightarrow X$ such that $\mathbb{I}_f$
belongs to $\gamma(L^2(S);X)$. We endow it with $\norme{f}_{\gamma(S;X)}:= 
\norme{\mathbb{I}_f}_{\gamma(L^2(S);X)}$. A remarkable fact is the density 
of simple function in the set $\gamma(S;X)$ \cite[Proposition 9.2.5]{hnvw}.

Now we collect some important results, which will be useful in the next sections.
We start with the so-called Multiplier Theorem \cite[Theorem 9.5.1]{hnvw},
a high ranking result involving the $\gamma$-boundedness. 
We state it under the assumption that $X$ does not contain $c_0$.
Thus the following statement applies to $K$-convex spaces.  

\begin{thm}[$\gamma$-Multiplier theorem]\label{multithm} 
Let $X$ be a Banach space not containing  $c_0$. 
Let $M : S \rightarrow \mathcal{L}(X)$ be a strongly mesurable function
and assume that its range $\mathcal{M} := \{ M(s) : s\in S\}$ is $\gamma$-bounded. 
Then for every function $\psi : S \rightarrow X$ in $\gamma(S;X)$, 
the function $M\psi : S \rightarrow X$ belongs to $\gamma(S;X)$, and we have
\[
\norme{M\psi}_{\gamma(S;X)} \leq \gamma(\mathcal{M})\norme{\psi}_{\gamma(S;X)}.
\]
\end{thm}

The next result is an inequality of Hölder type \cite[Theorem 9.2.14 (1)]{hnvw}.  
\begin{thm}[$\gamma$-Hölder inequality]\label{inetrace}
If $f: S \rightarrow X$ and $g : S \rightarrow X^*$ belongs to 
$\gamma(S;X)$ and $\gamma(S;X^*)$, respectively, then $\langle f,g\rangle$ belongs to $L^1(S)$ and we have
\[
\norme{\langle f,g\rangle}_{L^1(S)} \leq \norme{f}_{\gamma(S;X)}\norme{g}_{\gamma(S;X^*)}.
\]
\end{thm}

Now we give an extension result, for which we refer to \cite[Theorem 9.6.1]{hnvw}.
We identify the algebraic tensor product $H^*\otimes X$ with the space 
of finite rank bounded operator operators from $H$ into $X$ in the usual way, that
is, we set $(h^*\otimes x)(h) = h^*(h)x$ for any $h\in H,\, h^*\in H^*$
and $x\in X$.

\begin{thm}\label{extensionth}[Extension theorem]
Let $H_1$ and $H_2$ be Hilbert spaces. For any bounded operator $U : H_1^* \rightarrow H_2^*$, the mapping
\[
U\otimes I_X : H_1^*\otimes X\longrightarrow H_2^*\otimes X,
\]
taking $h^*\otimes x$ to $U(h^*)\otimes x$ for any $h^*\in H_1^*$ and $x\in X$,
has an unique extension to a bounded linear operator 
$\tilde{U} : \gamma(H_1;X) \rightarrow \gamma(H_2;X)$ of the same norm. 
Furthermore for all $T \in \gamma(H_1;X)$, 
\begin{equation}\label{extensioneq}
\tilde{U}(T) = T\circ \,^t U,
\end{equation}
where $^tU$ denotes the Banach space adjoint of $U$.
\end{thm}

To conclude this part, we apply the above principles to 
the Fourier-Plancherel transform $L^2(\R)\to L^2(\R)$. We identify the dual of $L^2(\R)$
with $L^2(\R)$ via the usual duality map provided by integration on $\R$.
   
\begin{lem}\label{fourierlem}
For any $f\in L^1(\R;X)$, let $\hat{f}$ be its Fourier transform defined by 
\[
\hat{f}(t) = \,\int_{\R}e^{-its}f(s)ds.
\]

Let  $\mathcal{F} : L^2(\R) \rightarrow L^2(\R)$ be the Fourier-Plancherel transform 
(which coincides with $f \mapsto \hat{f}$ on $L^1(\R)\cap L^2(\R)$). Let  
$\tilde{\mathcal{F}} : \gamma(L^2(\R);X) \rightarrow \gamma(L^2(\R);X)$ 
be its extension provided by Theorem \ref{extensionth}. If $f\in \gamma(\R;X) \cap L^1(\R;X)$, then we have
\[
\widehat{f} \in \gamma(\R;X) 
\qquad\hbox{and}\qquad
\tilde{\mathcal{F}}(\mathbb{I}_{f}) = \mathbb{I}_{\widehat{f}}
\]
and further,
\[
\norme{\widehat{f}}_{\gamma(\R;X)} = \sqrt{2\pi}\,\norme{f}_{\gamma(\R;X)}.
\]
\end{lem}
\begin{proof}
Obviously $\widehat{f}$ is measurable, and as $f$ is weakly
$L^2$, $\widehat{f}$ is also weakly $L^2$. Indeed for $x \in X^*$ one has by Fourier-Plancherel theorem :
\begin{align*}
\norme{\langle \widehat{f},x^* \rangle}^2_{L^2(\R)} &= 
\int_{\R}\big|\langle\int_{\R}f(s)e^{-ist}ds,x^*\rangle\big|^2dt  = 
\int_{\R}\big|\int_{\R}e^{-its}\langle f(s),x^*\rangle ds\big|^2dt \\
& = \norme{\mathcal{F}(\langle f,x^*\rangle)}^2_{L^2(\R)} = 
2\pi\,\norme {\langle f,x^*\rangle}^2_{L^2(\R)}.
\end{align*}
It follows that $\mathbb{I}_{\widehat{f}}$ is well defined and bounded. 
Now let $g\in L^2(\R)\cap L^1(\R)$. Then by Fubini theorem, using equality 
$^t\mathcal{F} = \mathcal{F}$ and \eqref{extensioneq}:
\begin{align*}
\tilde{\mathcal{F}}(\mathbb{I}_f)(g) &= \mathbb{I}_f\circ \mathcal{F}(g) =
\int_{\R}\mathcal{F}(g)(t)f(t)dt = 
\int_{\R}\big(\int_{\R}g(s)e^{-ist}ds\big)f(t)dt  \\ & = \int_{\R}g(s)\big(\int_{\R}f(t)e^{-its}dt\big)ds  = \mathbb{I}_{\widehat{f}}(g). 
\end{align*}
By density and since $\tilde{\mathcal{F}}(\mathbb{I}_f)$ and $\mathbb{I}_{\widehat{f}}$
are bounded,  the equality $\tilde{\mathcal{F}}(\mathbb{I}_f) = \mathbb{I}_{\widehat{f}}$
follows. Hence $\widehat{f} \in \gamma(\R;X)$.

Finally since $(\sqrt{2\pi})^{-1}\mathcal{F}$ is an isometry,
extension principle yields the equalities
\[
\norme{\widehat{f}}_{\gamma(\R;X)} = \norme{\mathbb{I}_{\widehat{f}}}_{\gamma(L^2(\R);X)} = 
\norme{\tilde{\mathcal{F}}(\mathbb{I}_f)}_{\gamma(L^2(\R);X)} =  
\sqrt{2\pi} \norme{\mathbb{I}_f}_{\gamma(L^2(\R);X)} = \sqrt{2\pi}\norme{f}_{\gamma(\R;X)}.
\]

\end{proof}


\section{STRONG $\gamma$-m-BOUNDED FUNCTIONAL CALCULUS}

We will say that a $C_0$-semigroup $(T_t)_{t\geq 0}$ on Banach space is of $\gamma$-type $\omega$ (resp.
of $W\gamma$-type $\omega$) if the set $\{ e^{\omega t}T_t\, :\, t\geq 0 \}$ is $\gamma$-bounded (resp.
$W\gamma$-bounded). If no such $\omega$ exists, we will say that $(T_t)_{t\geq 0}$ has no $\gamma$-type
(resp. no $W\gamma$-type).

\begin{ex}\label{nogtype}
It is easy to exhibit $C_0$-semigroups with no $\gamma$-type. Let $1\leq p<\infty$ and let 
$(S_t)_{t\geq 0}$ be the right translation $C_0$-semigroup on $L^p(\R)$, defined by
$S_t(f) = f(\cdot-t)$ for any $t\geq 0$ and any
$f\in L^p(\R)$. Then, for $p \not=2$, $(S_t)_{t\geq 0}$ has no $\gamma$-type. This
follows from the well-known fact that 
$\{S_t, \; t\in [0,1]\}$ is not $\gamma$-bounded. 
Indeed, assume that $p \in [1, 2)$ and for $n\in \N$ and $i \in \{1, \ldots,n\}$, 
let $t_i^n = \frac{i-1}{n}$ and $f^n_i = \chi_{[0,1/n]}$. Then  we have
\[
\norme{\big(\sum_{i=1}^{n}  |S_{t_i^n}f^n_i|^2\big)^{1/2}}_{L^p(\R)} = \norme{\chi_{[0,1]}}_{L^p(\R)} = 1.
\]
whereas
\[
\norme{\big(\sum_{i=1}^{n}  |f^n_i|^2\big)^{1/2}}_{L^p(\R)} = \norme{n^{1/2}\chi_{[0,1/n]}}_{L^p(\R)}  = n^{1/2 - 1/p}.
\]
Hence the inequality \eqref{Rboundedness2} cannot be true. The proof in the case $p \in (2, \infty)$ 
is similar.

The fact that the set $\{S_t\, :\,  t\in [0,1]\}$ is not $\gamma$-bounded immediately implies that for any $\omega$,
the set $\{e^{\omega t}S_t\, :\,  t\in [0,1]\}$ is not $\gamma$-bounded, and hence 
$\{e^{\omega t}S_t\, :\,  t\geq 0\}$ cannot be $\gamma$-bounded.
\end{ex}

Throughout this section, we let $X$ be a Banach space. Then we let $A$ be a half-plane type operator of type 
$\omega$ on $X$.

The condition that $A$ has a bounded $H^{\infty}$-functional calculus of type $\omega$ 
can be rephrased by saying that the set 
\[
\underset{\alpha < \omega }\bigcup \{f(A)\, :\, \, f \in H^{\infty}(R_{\alpha}), 
\norme{f}_{H^{\infty}(R_{\omega})} \leq 1 \} 
\]
is uniformly bounded.
This formulation and Definition \ref{mbounded} motivate the following definitions.

\begin{df}
\ 
\begin{enumerate}[label = \arabic*)]
\item
We say that $A$ has a $\gamma$-bounded (resp. a  $W\gamma$-bounded) 
$H^{\infty}$-functional calculus of type $\omega$ if $A$ 
has a bounded 
$H^{\infty}$-functional calculus of type $\omega$ and the set 
\[
\underset{\alpha < \omega }\bigcup \{f(A)\,:\, 
f \in H^{\infty}(R_{\alpha}), \norme{f}_{H^{\infty}(R_{\omega})} \leq 1 \} 
\]
is $\gamma$-bounded (resp. $W\gamma$-bounded).
\item
Let $m\geq 1$ be an integer.
We say that $A$ has a strong $\gamma$-$m$-bounded (resp. a
strong $W\gamma$-$m$-bounded)
functional calculus 
of type $\omega$ if $A$ has a strong $m$-bounded functional calculus 
of type $\omega$ and the set 
\begin{equation}\label{E}
\underset{\alpha < \omega }\bigcup \{(\omega - \alpha)^mf^{(m)}(A)\, :\,  
f \in H^{\infty}(R_{\alpha}), \norme{f}_{H^{\infty}(R_{\alpha})} \leq 1 \} 
\end{equation}
is $\gamma$-bounded (resp. $W\gamma$-bounded).
\end{enumerate}
\end{df}

\begin{rq1}\label{rq1}
\ 
\begin{enumerate}[label = \arabic*)]
\item 
To prove that the set 
\[\underset{\alpha < \omega }\bigcup \{(\omega - \alpha)^mf^{(m)}(A)\, :\,  f \in H^{\infty}(R_{\alpha}), 
\, \norme{f}_{H^{\infty}(R_{\alpha})} \leq 1 \} \text{ is $\gamma$-bounded,} 
\]
it is enough to prove that 
\[
\underset{\alpha < \omega }\bigcup \; \underset{\delta < \alpha }\bigcup \{(\omega - \alpha)^mf^{(m)}(A)   
\, :\,  f \in \mathcal{E}(R_{\delta}), \,\norme{f}_{H^{\infty}(R_{\alpha})} \leq 1 \} \text{ is $\gamma$-bounded.} 
	\]
This follows from the convergence lemma \cite[Theorem 3.1]{bat-haa}, the argument in the 
proof of \cite[Theorem 5.6 (a)]{bat-haa}, and Proposition \ref{R-summarizes}. Details
are left to the reader.
\item 
Likewise to prove that $A$ has a $\gamma$-bounded 
$H^{\infty}$-functional calculus of type $\omega$, it is enough to prove that the set
\[
\underset{\alpha < \omega }\bigcup 
\{f(A)\,:\, 
f \in \mathcal{E}(R_{\alpha}),\, \norme{f}_{H^{\infty}(R_{\omega})} \leq 1 \} 
\]
is $\gamma$-bounded.
\item 
If $A$ has a $\gamma$-bounded $H^{\infty}$-functional calculus of type $\omega$, then 
it has a strong $\gamma$-$m$-bounded functional calculus of type $\omega$.
Indeed consider the set
\[
\Delta := \underset{\beta < \omega }\bigcup
\{g(A)\, :\,  g \in H^{\infty}(R_{\beta}) ,\,\norme{g}_{H^{\infty}(R_{\omega})} \leq 1 \},
\]
and assume that $\Delta$ is $\gamma$-bounded.
Let $\alpha < \omega$ and let
$f \in H^{\infty}(R_{\alpha})$ with $\norme{f}_{H^{\infty}(R_{\alpha})} \leq 1$.
According to \cite[Lemma 5.4]{bat-haa}, $f^{(m)}\in H^{\infty}(R_{\beta})$ for any
$\beta\in(\alpha,\omega)$ and 
$$
\norme{\frac{(\omega - \alpha)^m}{m!}f^{(m)}}_{H^{\infty}(R_{\omega})} \leq 1.
$$ 
Consequently, 
\[
\underset{\alpha < \omega}\bigcup\bigg\{\frac{(\omega - \alpha)^m}{m!}
f^{(m)}(A)\, :\, f \in  H^{\infty}(R_{\alpha}), \norme{f}_{H^{\infty}(R_{\alpha})}\leq 1 \bigg\} \subset \Delta.
	\]
Hence the above set is $\gamma$-bounded, which shows that 
$A$ has a strong  $\gamma$-$m$-bounded functional calculus of type $\omega$. 
\end{enumerate}

The previous three statements hold as well with $W\gamma$-boundedness
replacing $\gamma$-boundedness.
\end{rq1}

Recall the condition $(GFS)_{m,\omega}$ from Definition \ref{GFS}. We introduce 
the following stronger form.

\begin{df}\label{WgammaGFS}
Let $m\geq 1$ be an integer.
We say that $A$ has property $(W\gamma\text{-}GFS)_{m,\omega}$ if there exists
a constant $C>0$ such that for any 
$N \in \N$, for any  $\alpha_1, \ldots , \alpha_N < \omega$, and for any
$x_1, \ldots , x_N \in X$ and $y_1, \ldots , y_N \in X^*$, we have
\[	 
\sum_{k=1}^{N}\int_{\R}|\langle (\omega-\alpha_k)^m R(\alpha_k+it,A)^{m+1}x_k,y_k \rangle| dt 
		\leq C\norme{\sum_{k=1}^N\gamma_k\otimes x_k}_{G(X)}\norme{\sum_{k=1}^N\gamma_k\otimes y_k}_{G(X^*)}.
\]
\end{df}

Clearly, $(W\gamma\text{-}GFS)_{m,\omega}$ implies $(GFS)_{m,\omega}$. Further if $X$ is a Hilbert space, 
then $(GFS)_{m,\omega}$ and $(W\gamma\text{-}GFS)_{m,\omega}$ are equivalent.

In the sequel we let $\mathbb{D}=\{ z\in\C\, :\, \vert z\vert<1\}$ denote the open unit disc.
The following statement is straightforward.

\begin{lem}\label{equiWGFS}
Let $m\geq 1$ be an integer. Assume that  $A$ has property $(GFS)_{m,\omega}$.
For any measurable function $\epsilon : \R \rightarrow \overline{\mathbb{D}}$ 
and for any $\alpha < \omega$, let 
$$
\int_{\R}\epsilon(t)R(\alpha+it,A)^{m+1}dt\ \in \mathcal{L}(X,X^{**})
$$
denote the operator defined by 
\[
\Bigl\langle\Bigl( \int_{\R}\epsilon(t)R(\alpha+it,A)^{m+1} dt\Bigr)x,y \Bigr\rangle \,=\, 
\int_{\R}\langle\epsilon(t)R(\alpha+it,A)^{m+1}x,y \rangle dt,\qquad  x \in X,\ y \in X^*.
\]
Then, the operator $A$ has property $(W\gamma\text{-}GFS)_{m,\omega}$  if and only if the set 
\begin{equation}\label{Set}
\big\{(\omega-\alpha)^m\displaystyle\int_{\R}\epsilon(t)R(\alpha+it,A)^{m+1}dt \, :\,
\epsilon : \R \rightarrow \overline{\mathbb{D}}\text{ measurable},\, \alpha < \omega \big\}
\end{equation}
is $W^*\gamma$-bounded.
\end{lem}

\begin{rq1}

Arguing as in Remark \ref{sechaveGFS}, one shows that if 
$A$ is $\gamma$-sectorial of $\gamma$-type
$<\frac{\pi}{2}$, then 
$A$ has property $W\gamma$-$(GFS)_{1,0}$. (We refer e.g. to  
\cite{hnvw} for information on $\gamma$-sectorial operators.)

Indeed assume that $A$ is $\gamma$-sectorial of $\gamma$-type
$<\frac{\pi}{2}$. Then the set
$$
\bigl\{(\alpha+it)R(\alpha +it,A)\, :\, \alpha<0,\ t\in\R\bigr\}
$$
is $\gamma$-bounded. Next for any measurable
$\epsilon : \R \rightarrow \overline{\mathbb{D}}$, we can write
$$
(-\alpha)\int_{\R}\epsilon(t)R(\alpha+it,A)^2 dt\,
= \int_{\R} (-\alpha)(\alpha+it)^{-2}
\epsilon(t)\bigl((\alpha+it)R(\alpha +it,A)\bigr)^2 dt.
$$
Since $\norme{t\mapsto (-\alpha)(\alpha+it)^{-2}
\epsilon(t)}_{L^1}\leq \pi$, Lemma \ref{Linftymeans} ensures that
the set
$$
\big\{(-\alpha)\displaystyle\int_{\R}\epsilon(t)R(\alpha+it,A)^{2}dt \, :\,
\epsilon : \R \rightarrow \overline{\mathbb{D}}\text{ measurable},\, \alpha < 0 \big\}
$$
is $W^*\gamma$-bounded. Hence the result follows from the above Lemma \ref{equiWGFS}.
\end{rq1}

\begin{prop}\label{propPm}
Let $m\geq 1$ be an integer and
assume that $A$ has property $(W\gamma\text{-}GFS)_{m,\omega}$. 
Then for any integer $1\leq p \leq m$, 
$A$ has property $(W\gamma\text{-}GFS)_{p,\omega}$.
\end{prop}

\begin{proof}
We proceed by induction, showing that if $m\geq 2$, then $(W\gamma\text{-}GFS)_{m,\omega}$
implies $(W\gamma\text{-}GFS)_{m-1,\omega}$.

Suppose that $A$ has property $(W\gamma\text{-}GFS)_{m,\omega}$, with $m\geq 2$. 
Let $\epsilon : \R \rightarrow \overline{\mathbb{D}}$ be a measurable function
and let $\alpha < \omega$. Applying \cite[Proposition 6.3. (a)]{bat-haa}, we have
$$
R(\alpha+it,A)^{m}\,=\, -m \int_{-\infty}^{\alpha}R(u+it,A)^{m+1}du
$$
for any $t\in\R$. Hence
for any $x\in X$, for any $y \in X^*$ and for any $t \in \R$,
\begin{align*}
\langle \epsilon(t) R(\alpha+it,A)^{m}x,y \rangle 
& = -\int_{-\infty}^{\alpha} m\langle  \epsilon(t)R(u+it,A)^{m+1}x,y \rangle du\\
& = -\int_{-\infty}^{0} m\langle  \epsilon(t)R(\alpha + u+it,A)^{m+1}x,y \rangle du.
\end{align*}
We now integrate over $t$. Property $(GFS)_{m,\omega}$ ensures that 
we can apply Fubini's theorem in the following
computation:
\begin{align*}
(\omega-\alpha)^{m-1} \int_{\R}&\langle  \epsilon(t)R(\alpha+it,A)^{m}x,y \rangle dt \\
&=- \int_{\R}m(\omega-\alpha)^{m-1}\int_{-\infty}^{0} 
\langle\epsilon(t)R(	\alpha + u+it,A)^{m+1}x,y \rangle du dt \\ 
&=  -\int_{-\infty}^{0}m(\omega-\alpha)^{m-1}\int_{\R} 
\langle\epsilon(t)R(	\alpha + u+it,A)^{m+1}x,y \rangle dt du\\
&=- \int_{-\infty}^{0}m\frac{(\omega-\alpha)^{m-1}}{(\omega -\alpha -u )^m}
\int_{\R}\langle(\omega -\alpha-u)^m \epsilon(t)R(	\alpha + u + it,A)^{m+1}x,y \rangle dt du.
\end{align*}

Let $\mathcal{T} $ be the set (\ref{Set}). By assumption, $\mathcal{T} $ is $W^*\gamma$-bounded hence by
Lemma \ref{WeakL1means},
the set
\begin{align*}
\Gamma := \big\{ \int_{-\infty}^{0} \phi(u)\int_{\R} &(\omega -\alpha-u)^m 
\epsilon(t)R(	\alpha + u + it,A)^{m+1} dt du\, :\, \\
& \epsilon : \R \rightarrow \overline{\mathbb{D}}\text{ measurable},\, \alpha < \omega,\,
\phi\in L^1((-\infty,0)),\, \norme{\phi}_{L^1} \leq 2 \big\}
\end{align*}
is $W^*\gamma$-bounded. Since 
$$
\norme{m\frac{(\omega-\alpha)^{m-1}}{(\omega-\alpha-\,\cdot)^{m}}}_{L^1((-\infty,0))} = \frac{m}{m-1} \leq 2,
$$ 
the above calculation shows that the set
\[
\big\{(\omega-\alpha)^{m-1}\displaystyle\int_{\R}\epsilon(t)R(\alpha+it,A)^{m}dt
\, :\, \epsilon : \R \rightarrow \overline{\mathbb{D}}\text{ measurable},\, \alpha < \omega \big\}  
\]
is included in $\Gamma$, hence is $W^*\gamma$-bounded. Hence, by Lemma \ref{equiWGFS}, 
the operator $A$ has  property $(W\gamma\text{-}GFS)_{m-1,\omega}$. 
\end{proof}

We recalled that $(GFS)_{m,\omega}$ is equivalent 
to strong $m$-bounded functional calculus of type $\omega.$  The following theorem 
provides a similar statement in the context of $W\gamma$-boundedness.

\begin{thm}\label{thmeqWgamma}
The following assertions are equivalent for $m\geq 1$:
\begin{enumerate}[label = (\roman*)]
\item  $A$ has $(W\gamma\text{-}GFS)_{m,\omega}$;
\item $A$ has a strong $W\gamma$-$1$-bounded functional calculus of type $\omega$;
\item $A$ has a strong $W\gamma$-$m$-bounded functional calculus of type $\omega$.
\end{enumerate}

Moreover, if $A$ satisfies these conditions, then $-A$ generates a $C_0$-semigroup of $W\gamma$-type $\omega$.
\end{thm}

\begin{proof}
\ 

$(i) \Rightarrow (ii)$: First, by Proposition \ref{propPm}, $A$ has $(W\gamma\text{-}GFS)_{1,\omega}$.

Let $\alpha_1,\ldots,\alpha_N <\omega$, 
$\delta_1<\alpha_1,\ldots,\delta_N<\alpha_N$, and $f_1 \in \mathcal{E}(R_{\delta_1}),\ldots,f_N\in 
\mathcal{E}(R_{\delta_N})$ with $\norme{f_k}_{H^{\infty}(R_{\alpha_k})}\leq 1$. Let 
$x_1,\ldots,x_N \in X$ and $y_1,\ldots,y_N \in X^*$. It follows from the proof of 
\cite[Theorem 5.6 (a)]{bat-haa}) that for any $k=1,\ldots,N$,
\[
f'_k(A) = \frac{-1}{2\pi}\int_{\R}f(\alpha_k + it)R(\alpha_k+it,A)^{2}dt.
\]
Consequently,
\begin{align*}
&\sum_{k=1}^N |\langle(\omega-\alpha_k)f'_k(A)x_k,y_k\rangle| \\
&= \sum_{k=1}^N \big|\bigl\langle\frac{-1}{2\pi}\Bigl(\int_{\R}(\omega-\alpha_k)
f_k(\alpha_k + it)R(\alpha_k+it,A)^{2} dt\Bigr) x_k,y_k\bigr\rangle\big| \\
&\leq  \frac{1}{2\pi} \sum_{k=1}^N \int_{\R} (\omega-\alpha_k)\vert f_k(\alpha_k +it)\vert 
\vert \langle R(\alpha_k+it,A)^{2}x_k,y_k\rangle\vert dt\\
&\leq  \frac{1}{2\pi}\sum_{k=1}^N \norme{f_k}_{H^{\infty}(R_{\alpha_k})}
\int_{\R}\big|\langle(\omega-\alpha_k)R(\alpha_k+it,A)^{2}x_k,y_k\rangle\big|dt.  \\
\end{align*}
Since $\norme{f_k}_{H^{\infty}(R_{\alpha_k})} \leq 1$ and $A$ has $(W\gamma\text{-}GFS)_{1,\omega}$,  
this yields an estimate 
\[	
\sum_{k=1}^N |\langle(\omega-\alpha_k)f'_k(A)x_k,y_k\rangle| 
\leq C\norme{\sum_{k=1}^N\gamma_k\otimes x_k}_{G(X)}\norme{\sum_{k=1}^N\gamma_k\otimes y_k}_{G(X^*)}.
\]
According to Remark \ref{rq1} (1),
this implies that $A$ has a strong $W\gamma$-1-bounded functional calculus of type $\omega$.
	 

\smallskip
$(ii) \Rightarrow (iii)$: It follows from the assumption that the set 
\[
\Delta := \underset{\alpha <\beta < \omega }\bigcup
\big\{(\omega - \beta)g^{(m)}(A)\, :\, g \in H^{\infty}(R_{\alpha}),\, 
\norme{g^{(m-1)}}_{H^{\infty}(R_{\beta})} \leq 1 \big\}
\]
is $W\gamma$-bounded. For any $\alpha <\beta < \omega$ and $f \in H^{\infty}(R_{\alpha})$
with $\norme{f}_{H^{\infty}(R_{\alpha})} \leq 1$, we have
$$
\norme{\frac{(\beta - \alpha)^{m-1}}{(m-1)!}f^{(m-1)}}_{H^{\infty}(R_{\beta})} \leq 1
$$
by \cite[Lemma 5.4]{bat-haa}. Hence
\[
\underset{\alpha<\beta<\omega}\bigcup\bigg\{\frac{(\omega - \beta)
(\beta - \alpha)^{m-1}}{(m-1)!}f^{(m)}(A), \;f \in  H^{\infty}(R_{\alpha}), 
\norme{f}_{H^{\infty}(R_{\alpha})}\leq 1 \bigg\} \subset \Delta. 
\]
Taking $\beta = \frac{\alpha + \omega}{2}$ in the above set, we obtain 
\[
\underset{\alpha < \omega}\bigcup\bigg\{\frac{(\omega - \alpha)^m}{2^m (m-1)!}f^{(m)}(A)
\, :\, f \in  H^{\infty}(R_{\alpha}), \norme{f}_{H^{\infty}(R_{\alpha})} \leq 1 \bigg\} \subset \Delta.
\]
Hence the above is $W\gamma$-bounded.
Thus $A$ has  a strong $W\gamma$-$m$-bounded functional calculus of type $\omega$.


\smallskip
$(iii) \Rightarrow (i)$: Let $\alpha_1,\ldots,\alpha_N< \omega$, 
$x_1, \ldots, x_N \in X$ and  $y_1,\ldots, y_N \in X^{*}$. Arguing as in
the proof of \cite[Theorem 5.6 (b)]{bat-haa}, we consider $\beta_k\in(\alpha_k,w)$ for 
any $k=1,\ldots, N$ and introduce measurable functions
$\epsilon_1, ..., \epsilon_N : \R \to \overline{\mathbb{D}}$ such that  
\[
|\langle R(\alpha_k+it,A)^{m+2}x_k,y_k\rangle| = 	\langle R(\alpha_k+it,A)^{m+2}x_k,y_k \rangle \epsilon_k(t)
\]
for all $k=1,\ldots, N$ and all $t \in \R$. Next for any $R>0$ and any $Re(z) > \beta_k$, we set
\[
\phi_{k,R}(z) := \frac{\beta_k -\alpha_k}{\pi}\int_{-R}^{R}\frac{\epsilon_k(t)}{(\alpha_k+it-z)^2}dt
\]
Since $R$ is finite, it is easy to show that $\phi_{k,R}(z) = O(|z|^{-2})$ as $|z| \rightarrow \infty$, and hence
$\phi_{k,R} \in \mathcal{E}(R_{\alpha})$. Furthermore, it follows from
the proof of \cite[Theorem 5.6 (b)]{bat-haa}
that $\norme{\phi_{k,R}}_{H^{\infty}(R_{\beta_k})} \leq 1$ and
\[
\phi_k^{(m)}(A) = \int_{-R}^{R} \frac{\beta_k -\alpha_k}{\pi}(m+1)!\epsilon_k(t)R(\alpha_k +it,A)^{m+2}dt.
\]
It therefore follows from $(iii)$ that we have an estimate
\begin{align*}
\sum_{k=1}^{N} & \int_{-R}^{R}\big|\langle \frac{(\beta_k -\alpha_k)(\omega - \beta_k)^m}{\pi}(m+1)! R(\alpha_k+it,A)^{m+2}x_k,y_k \rangle\big|dt \\
& = \sum_{k=1}^{N} \big|  \big\langle (\omega - \beta_k)^m\phi_{k,R}^{(m)}(A)x_k,y_k \big\rangle \big| \\
& \leq C\norme{\sum_{k=1}^N\gamma_k\otimes x_k}_{G(X)}\norme{\sum_{k=1}^N\gamma_k\otimes y_k}_{G(X^*)}.
\end{align*}
Passing to the limit when $R\to\infty$, one obtains
\begin{align*}
&\sum_{k=1}^{N}  \int_{\R}\big|\langle \frac{(\beta_k -\alpha_k)(\omega - \beta_k)^m}{\pi}(m+1)! 
R(\alpha_k+it,A)^{m+2}x_k,y_k \rangle\big|dt \\
&	\leq C\norme{\sum_{k=1}^N\gamma_k\otimes x_k}_{G(X)}\norme{\sum_{k=1}^N\gamma_k\otimes y_k}_{G(X^*)}.
\end{align*}
Now we choose $\beta_k =  \frac{\omega + \alpha_k}{2}$ in the above estimate.
We obtain the following inequality
\begin{align*}
&\sum_{k=1}^{N} \int_{\R}|\langle (\omega -\alpha_k)^{m+1}
R(\alpha_k+it,A)^{m+2}x_k,y_k \rangle|dt \\
&\leq \frac{2^{m+1}C\pi}{(m+1)!}\norme{\sum_{k=1}^N\gamma_k\otimes x_k}_{G(X)}
\norme{\sum_{k=1}^N\gamma_k\otimes y_k}_{G(X^*)}.
\end{align*}	
This shows that $A$ has $(W\gamma\text{-}GFS)_{m+1,\omega}$.
Then by Proposition \ref{propPm}, $A$ has $(W\gamma\text{-}GFS)_{m,\omega}$.

\smallskip
Finally, assume that $(ii)$ holds true. 
In particular $A$ has a strong $m$-1-bounded functional calculus of type $\omega$ hence 
by Theorem \cite[Theorem 6.4]{bat-haa}, $-A$ generates a $C_0$-semigroup $(T_t)_{t\geq 0}$ of type $\omega$. 

For any $t\geq 0$ and $\alpha <\omega $, let $\phi_{\alpha,t}(z) = e^{-tz}e^{t\alpha}$ for $z\in R_\alpha$.
Then $\norme{\phi_{\alpha,t}}_{H^{\infty}(R_{\alpha})} = 1$. Hence 
by $(ii)$, the set 
\[
\displaystyle\underset{\alpha <\omega}\bigcup\big\{ (\omega-\alpha)\phi'_{\alpha,t}(A)
\, :\, t\geq 0 \big\} \text{ is W$\gamma$-bounded}.
\]	 
Therefore the set 
\[
\displaystyle\underset{\alpha < \omega}\bigcup\big\{ (\omega-\alpha)(-t)e^{-tA}e^{t\alpha}
\, :\, t> 0 \big\} \text{ is $W\gamma$-bounded}. 
\]
We noticed in Proposition \ref{propexp} that $e^{-tA}=T_t$ for any $t\geq 0$. 
Hence taking $\alpha = \omega-\frac{1}{t}$ for any $t>0$, we deduce that the set
\[
\{e^{\omega t}T_t \, :\, t\geq 0\} \text { is $W\gamma$-bounded.}
\]
Hence $(T_t)_{t\geq 0}$ is a $C_0$-semigroup of $W\gamma$-type $\omega$.
\end{proof}

\begin{rq1}\label{Add}
The above proof shows as well that 
if $A$ has a strong $\gamma$-1-bounded functional calculus of type $\omega$,
then $-A$ generates a $C_0$-semigroup of $\gamma$-type $\omega$.
\end{rq1}

In the $K$-convex case, Theorem \ref{thmeqWgamma} can be strengthened as follows.

\begin{thm}\label{thmstrWgammaeq}
Assume that $X$ is $K$-convex and let $m\geq 1$ be an integer. The following assertions are equivalent:
\begin{enumerate}[label = (\roman*)]
\item  
There exist a norming subspace $Z\subset X^*$ and a constant $C>0$
such that for any $N \in \N$, for any $\alpha_1, \ldots , \alpha_N < \omega$,
for any $x_1, \ldots ,x_N  \in  X$ and for any $y_1, \ldots , y_N \in Z$,
\[	 
\sum_{k=1}^{N}\int_{\R}|\langle (\omega-\alpha_k)^m R(\alpha_k+it,A)^{m+1}x_k,y_k \rangle| dt 
\leq C\norme{\sum_{k=1}^N\gamma_k\otimes x_k}_{G(X)}\norme{\sum_{k=1}^N\gamma_k\otimes y_k}_{G(X^*)}.
\]
\item 
$A$ has a strong $\gamma$-$1$-bounded functional calculus of type $\omega$;
\item 
$A$ has a strong $\gamma$-$m$-bounded functional calculus of type $\omega$.
\end{enumerate}
\end{thm}

\begin{proof}
The proofs of $(ii) \Leftrightarrow (iii)$ and $(iii) \Rightarrow (i)$ are obvious by the 
equivalence of $W\gamma$-boundedness and $\gamma$-boundedness on a $K$-convex space, and
Theorem \ref{thmeqWgamma}.

Now assume $(i)$. By \cite[Corollary 7.4.6.]{hnvw}, there exists $M>0$ 
such that for all $N \in \N$ and $x_1, \ldots, x_N\in X$ , we have
\begin{equation}\label{Kconv_norming_ine}
\norme{\sum_{k=1}^N \gamma_k \otimes x_k}_{G(X)} \leq M 
\sup\bigg\{ \big| \sum_{k=1}^N\langle x_k, y_k \rangle 
\big|, \; y_k\in Z,\, \norme{\sum_{k=1}^N \gamma_k \otimes y_k }_{G(X^*)} \leq 1 \bigg\}.
\end{equation}
The argument in the proof of Theorem \ref{thmeqWgamma}
shows that for all $\alpha_1,\ldots,\alpha_N <\omega$, $\delta_1<\alpha_1,\ldots,\delta_N<\alpha_N$, 
$f_1 \in \mathcal{E}(R_{\delta_1}),\ldots,f_N\in \mathcal{E}(R_{\delta_N})$ with  $\norme{f_k}_{H^{\infty}(R_{\alpha_k})}
\leq 1$, $x_1,\ldots,x_N \in X$ and $y_1,\ldots,y_N \in Z$, we have an estimate
\[	
\sum_{k=1}^N |\langle(\omega-\alpha_k)f'_k(A)x_k,y_k\rangle| 
\leq C\norme{\sum_{k=1}^N\gamma_k\otimes x_k}_{G(X)}\norme{\sum_{k=1}^N\gamma_k\otimes y_k}_{G(X^*)}.
\]
Applying \eqref{Kconv_norming_ine}, this implies
$$
\norme{\sum_{k=1}^N \gamma_n \otimes (\omega-\alpha_k)f'_k(A)x_k}_{G(X)} 
\leq MC\norme{\sum_{k=1}^N\gamma_k\otimes x_k}_{G(X)}.
$$
Hence $A$ has a $\gamma$-$1$-bounded functional calculus of type $\omega$.
\end{proof}


\section{A SHI-FENG-GOMILKO THEOREM ON $K$-CONVEX SPACES}
Let $(\Omega,\mu)$ be a measure space, let $N \in \N$ and let $(e_1,\ldots, e_N)$ 
denote the canonical basis of $l^{2}_N$. Let $l^{2}_N \overset{2}\otimes L^2(\Omega)$
denote the Hilbert space tensor product of $l^2_N$ and $L^2(\Omega)$. With the
notation $\N_N = \{1, \ldots, N \}$,
we have a unitary isomorphism
$$
l^{2}_N \overset{2}\otimes L^2(\Omega) = L^2(\Omega\times \N_N).
$$
Let $X$ be a Banach space.
For any bounded operator $u : l^{2}_N \overset{2}\otimes L^2(\Omega) \longrightarrow X$, let 
$u_k : L^{2}(\Omega) \rightarrow X$ be defined by $u_k(f) = u(e_k \otimes f)$, for any $k=1, \ldots, N$. 
Then the mapping $u  \mapsto \sum_ke_k\otimes u_k$ induces an algebraic isomorphism 
\[
\mathcal{L}(l^{2}_N \overset{2}\otimes L^2(\Omega), X) \simeq l^{2}_N\otimes \mathcal{L}(L^2(\Omega),X).
\]
It is easy to check that $u : l^2_N \otimes L^2(\Omega) \rightarrow X$ belongs 
to $\gamma(L^2(\Omega\times \N_N);X)$ if and only if $u_k$ belongs to $\gamma(L^2(\Omega);X)$ 
for any $k = 1, \ldots, N$. This leads to an algebraic isomorphism
\[
\gamma(L^2(\Omega\times \N_N);X) \simeq l^2_N \otimes \gamma(L^2(\Omega );X).
\]
Likewise, a function $f : \Omega \times \N_N \rightarrow X$ belongs to 
$\gamma(\Omega \times \N_N ; X)$ if and only if $f(\cdot,k)$ belongs to $\gamma(\Omega  ; X)$ for any $k = 1, \ldots, N$.

Recall the sun dual $X^{\odot}$ from Subsection \ref{Sun}. 
We now state and prove the main result of this article.

\begin{thm}\label{shifengRthm}
Let $X$ be a $K$-convex Banach space and let $m\geq 1$ be an integer. Let $A$
be an operator of half-plane type $\omega$ on $X$.
The following assertions are equivalent: 
\begin{enumerate}[label = (\roman*)]
\item 
$-A$ generates a $C_0$-semigroup $(T_t)_{t\geq0}$ of $\gamma$-type $\omega$;
\item 
There exists a constant $C>0$ such that for all $N\in \N$, 
for all $x_1,\ldots,x_N \in X$,  for all $y_1,\ldots,y_N \in \overline{dom(A^*)}= X^{\odot}$, and 
for all $\alpha_1, \ldots, \alpha_N < \omega$,
the functions 
$(t,k) \mapsto \sqrt{\omega-\alpha_k}R(\alpha_k + it, A)x_k$ and 
$(t,k) \mapsto \sqrt{\omega-\alpha_k}R(\alpha_k + it, A^*)y_k$ 
belong to $\gamma(\R\times\N_N;X)$ and $\gamma(\R\times\N_N;X^{*})$, respectively, and satisfy
\begin{equation}\label{ine1}
\norme{(t,k) \mapsto \sqrt{\omega-\alpha_k}R(\alpha_k + it, A)x_k}_{\gamma(\R\times\N_N;X)} \leq 
C\norme{\sum_{k=1}^{N}{\gamma_k\otimes x_k}}_{G(X)}
\end{equation}
and 
\begin{equation}
\label{ine1*}
\norme{(t,k) \mapsto \sqrt{\omega-\alpha_k}R(\alpha_k + it, A^*)y_k}_{\gamma(\R\times\N_N;X^{*})} 
\leq C\norme{\sum_{k=1}^{N}{\gamma_k\otimes y_k}}_{G(X^{*})};
\end{equation}
\item 
$A$ has a strong $\gamma$-$1$-bounded functional calculus of type $\omega$;
\item 
$A$ has a strong $\gamma$-$m$-bounded functional calculus of type $\omega$.
\end{enumerate}
\end{thm}

\begin{proof}
Without loss of generality, one may assume $\omega = 0$.

The equivalence $(iii) \Leftrightarrow (iv)$ and the
implication $(iv) \Rightarrow (i)$ follow from
Theorems \ref{thmeqWgamma} 
and \ref{thmstrWgammaeq}.

\smallskip
$(i)\Rightarrow (ii)$:
Since $\mathcal{T} := \{T_t\, :\, t\geq 0\}$ is $\gamma$-bounded, the strongly 
measurable function $M : \R_+\times\N_N \longrightarrow \mathcal{L}(X)$ defined 
by $M(t,k) := T_t$ has $\gamma$-bounded range $\mathcal{T}$. 
Then by Theorem \ref{multithm}, for each $\psi \in \gamma(\R_{+}\times\N_N;X)$,
the function $M\psi$ belongs to $\gamma(\R_{+}\times\N_N;X)$, with
\begin{equation}\label{ineMpsi}
\norme{M\psi}_{\gamma(\R_{+}\times\N_N;X)} 
\leq\gamma(\mathcal{T})\norme{\psi}_{\gamma(\R_{+}\times\N_N;X)}.
\end{equation}
Consider $\alpha_1, \ldots, \alpha_N <0$ and $x_1,\ldots,x_N\in X$, and define
$$
\psi(t,k) = \sqrt{-\alpha_k}e^{\alpha_kt}x_k,\qquad t\in\R_+,\  k=1,\ldots,N.
$$ 
Then $\psi \in \gamma(\R_+\times \N_N;X)$ and 
\begin{equation}\label{egapsi}
	\norme{\psi}_{\gamma(\R_+\times\N_N;X)} = \frac{1}{\sqrt{2}}\norme{\sum_{k=1}^{N}\gamma_k\otimes x_k}_{G(X)}.
\end{equation}
Indeed,  $\psi = \sum_{k=1}^N h_k\otimes x_k$ where $h_k : \R^+\times \N_N \rightarrow \R$ is defined by 
\[h_k\colon
(t,j)\longmapsto\begin{cases}
\sqrt{-\alpha_k}e^{\alpha_kt}&\text{if }j=k \\
0&\text{if } j\neq k.
\end{cases}
\]
Further $h_1,\ldots, h_N$ are pairwise orthogonal with 
$\norme{h_k}_{L^2(\R_+\times \N_N)}= \norme{\sqrt{-\alpha_k}e^{\alpha_kt}}_{L^2(\R_+)} = 
\frac{1}{\sqrt{2}}$. Hence \eqref{egapsi} follows from \cite[Example 9.2.4]{hnvw}.

Recall that for each $k \in \{ 1,\ldots, N\}$, we have
\[
R(\alpha_k +it,A)x_k = -\int_{0}^{\infty}e^{its}e^{\alpha_k s}T_sx_kds.
\]
Applying Lemma \ref{fourierlem}, we deduce that
$(t,k) \longmapsto R(\alpha_k+it,A)x_k$ belongs to $\gamma(\R_+\times\N_N;X)$, with
\begin{equation}\label{egaRfourier}
\norme{(t,k) \mapsto \sqrt{-\alpha_k}R(\alpha_k + it, A)x_k}_{\gamma(\R\times\N_N,X)} = 
\sqrt{2\pi}\norme{(t,k) \mapsto \sqrt{-\alpha_k}e^{\alpha_kt}T_tx_k}_{\gamma(\R\times\N_N,X)}. 
\end{equation}
Combining \eqref{ineMpsi}, \eqref{egapsi} and \eqref{egaRfourier}, we actually obtain
\[
\norme{(t,k) \mapsto \sqrt{-\alpha_k}R(\alpha_k + it, A)x_k}_{\gamma(\R\times\N_N,X)} 
\leq \sqrt{\pi}\gamma(\mathcal{T})\norme{\sum_{k=1}^{N}\gamma_k \otimes x_k}_{G(X)},
\]
which proves \eqref{ine1}.

Finally, since $X$ is $K$-convex, the set $\{T_t^*\, :\, t\geq0\}$ is $\gamma$-bounded.
Further $(T_t^{\odot})_{t\geq 0}$ is a $C_0$-semigroup on the sun dual $X^{\odot}$, with
generator equal to $-A^{\odot}$.
Then the above computations together with Theorem \ref{thmsun} lead to \eqref{ine1*}.

\smallskip
$(ii)\Rightarrow (iii)$: Let $\alpha_1,\ldots,\alpha_N <0$, $x_1, \ldots, x_N \in X$ and $y_1, \ldots, y_N \in X^{\odot}$.
Applying Theorem \ref{inetrace}, one obtains 
\begin{align*}
&\sum_{k=1}^{N}\int_{\R}|\langle -\alpha_k R(\alpha_k+it,A)^{2}x_k,y_k \rangle| dt  \\
& = \sum_{k=1}^{N}\int_{\R}|\langle \sqrt{-\alpha_k} R(\alpha_k+it,A)x_k,\sqrt{-\alpha_k}R(\alpha_k+it,A)^*y_k \rangle| dt \\
&=\norme{(t,k)\mapsto <\sqrt{-\alpha_k}R(\alpha_k+it,A)x_k, \sqrt{-\alpha_k}R(\alpha_k+it,A^*)y_k>}_{L^1(\R\times\N_N)}\\
& \leq \norme{(t,k)\mapsto \sqrt{-\alpha_k}R(\alpha_k+it,A)x_k}_{\gamma(\R\times\N_N;X)}
\norme{(t,k)\mapsto\sqrt{-\alpha_k}R(\alpha_k+it,A^*)y_k}_{\gamma(\R\times\N_N;X^*)} \\
&\leq C^2	\norme{\sum_{k=1}^{N}{\gamma_k\otimes x_k}}_{G(X)}\norme{\sum_{k=1}^{N}{\gamma_k\otimes y_k}}_{G(X^*)}.
\end{align*}
Since $X^{\odot}$ is norming in $X$, it follows from the above
estimate and Theorem \ref{thmstrWgammaeq} that $A$ has a strong $\gamma$-$1$-bounded calculus of type $0$.
\end{proof}

\begin{rq1}
Let $(S,\mu)$ be a measure space and let $E(S)$ be a $K$-convex
Banach function space over $(S,\mu)$
(see \cite[appendix F]{hnvw} for definition). Then 
$E(S)$ has finite cotype hence
according to \cite[Proposition 9.3.8]{hnvw}, there exist $c>0$ and $C>0$ such 
that for each $f\in \gamma(\R\times \N_N; E(S))$,
\begin{equation}\label{Lattice}
c\norme{f}_{E(S;L^2(\R\times \N_N))} \leq \norme{f}_{\gamma(\R\times \N_N;E(S))} \leq C \norme{f}_{E(S;L^2(\R\times \N_N))}.
\end{equation}
Furthermore, the following equality holds,
\[
\norme{f}_{E(S;L^2(\R\times \N_N))} = \norme{\Big(\int_{\R} \sum_{k=1}^N|f(\cdot,k)|^2\Big)^{\frac{1}{2}}}_{E(S)}.
\]
The space $E(S)^*$ satisfies similar properties.
Hence using (\ref{Lattice}) and the Khintchine-Maurey inequality 
\cite[Theorem 7.2.13]{hnvw}, the condition $(ii)$ of Theorem \ref{shifengRthm} can be replaced by:
\begin{itemize}
\item [$(ii)'$] There exists a constant $C>0$ such that for all $N\in \N$, 
for all $x_1,\ldots,x_N \in X$,  for all $y_1,\ldots,y_N \in \overline{dom(A^*)}= X^{\odot}$, and 
for all $\alpha_1, \ldots, \alpha_N < \omega$,
\begin{equation*}
\norme{\Big(\sum_{k=1}^{N}\int_{\R}(\omega-\alpha_k)|R(\alpha_k + it, A)x_k|^2dt \Big)^{\frac{1}{2}}}_{E(S)} 
\leq C\norme{\big(\sum_{k=1}^{N}|x_k|^2\big)^{\frac{1}{2}}}_{E(S)}
\end{equation*}
and 
\begin{equation*}	
\norme{\Big(\sum_{k=1}^{N}\int_{\R}(\omega-\alpha_k)|R(\alpha_k + it, A^*)y_k|^2dt \Big)^{\frac{1}{2}}}_{E(S)^*} 
\leq C\norme{\big(\sum_{k=1}^{N}|y_k|^2\big)^{\frac{1}{2}}}_{E(S)^*}.
\end{equation*}
\end{itemize}
Thus $-A$ generates a $C_0$-semigroup of $\gamma$-type $\omega$ on $E(S)$
if and only if $(ii)'$ holds true.

Of course the above applies when $E(S)=L^p(S)$ for some $1<p<\infty$.
\end{rq1}

\begin{rq1}\label{HR}
In \cite[Theorem 6.4]{haa-roz}, Haase and Rozendaal state that if $-A$ generates
a $\gamma$-bounded $C_0$-semigroup on a Banach space $X$, then $A$ 
has a strong $m$-bounded functional calculus of type $0$, for any 
$m\geq 1$. If $X$ is K-convex, this is a formal consequence of Theorem
\ref{shifengRthm} and in this case, the latter is a strengthening of 
the Haase-Rozendaal theorem.

For general $X$, a proof of \cite[Theorem 6.4]{haa-roz} can
be derived from the arguments in the proof of  Theorem
\ref{shifengRthm}. Indeed assume that 
$-A$ generates
a $\gamma$-bounded $C_0$-semigroup, consider $\alpha<0$ and 
let $x\in X$ and $y\in X^{\odot}$.
The proof of Theorem
\ref{shifengRthm} shows that $t\mapsto R(\alpha+it,A)x$ belongs to
$\gamma(\R;X)$ and that 
\begin{equation}\label{HR1}
\sqrt{-\alpha}\,\bigl\Vert   
t\mapsto R(\alpha+it,A)x\bigr\Vert_{\gamma(\R,X)}\leq C\norme{x}
\end{equation}
for some constant $C>0$ not depending either on $\alpha$ or $x$.
Then let $\gamma'(\R;X^*)$ be the space introduced
in \cite[Section 5]{kal-wei1}. Using \cite[Remark 5.12, (S2)]{kal-wei1}
instead of Theorem \ref{multithm}, one obtains in a similar manner that
\begin{equation}\label{HR2}
\sqrt{-\alpha}\,\bigl\Vert   
t\mapsto R(\alpha+it,A)^*y\bigr\Vert_{\gamma'(\R,X^*)}\leq C\norme{y}.
\end{equation}
Finally, using \cite[Remark 5.12, (S1)]{kal-wei1}
instead of Theorem \ref{inetrace}, one deduces from (\ref{HR1}) and (\ref{HR2}) that 
$$
(-\alpha)\int_{\R}\vert R(\alpha+it,A)^2x,y\rangle\vert dt\,\leq C^2\norme{x}\norme{y}.
$$
Since the sun dual $X^{\odot}$ is $w^*$-dense in $X^*$, this shows that 
$A$ 
has a strong $1$-bounded functional calculus of type $0$ (and hence
a strong $m$-bounded functional calculus of type $0$ for any $m\geq 1$).
\end{rq1}

\section{A GEARHART-PRÜSS THEOREM ON $K$-CONVEX SPACES}
Let $A$ be a half-plane type operator on some Banach space $X$.
Its abscissa of uniform boundedness $s_0(A)$ is defined by
\[
s_0(A) := \sup\big\{\alpha \in \R \, :\, \sigma(A)\subset R_\alpha\ \hbox{and}\ 
\underset{Re(z) \leq \alpha}\sup\norme{R(z,A)}< \infty \big\}.
\]
If $-A$ generates a $C_0$-semigroup $(T_t)_{t\geq 0}$, then the exponential growth bound $\omega(A)$ if defined 
as the supremum of all $\omega\in\R$ such that $(T_t)_{t\geq 0}$ is of type $\omega$, that is,
\[
\omega(A) := \sup \big\{\omega \in \R\,:\,
\text{there exists $M_{\omega}>0$ such that $ \norme{T_t} \leq M_{\omega}e^{-\omega t}$ for all } t \geq 0 \big \}.  
\]
We introduce $\gamma$-bounded analogues of these notions, as follows. First we set 
\[
s_0^{\gamma}(A) := \sup\big\{\alpha \in \R 
\, :\, \sigma(A)\subset R_\alpha\ \hbox{and}\ \text{the set }
\{R(z,A)\, :\, Re(z)\leq \alpha \} \text{ is $\gamma$-bounded}\big\},
\]
with the convention that $s_0^{\gamma}(A)=-\infty$ if no set $\{R(z,A)\, :\, Re(z)\leq \alpha \}$
is $\gamma$-bounded.
Second, if $-A$ generates a $C_0$-semigroup $(T_t)_{t\geq 0}$, and if the latter 
admits a $\gamma$-type, then we set
\[
\omega^{\gamma}(A) := \sup \big\{\omega \in \R\,:\,
\text{the set }\{ e^{\omega t}T_t\, :\, t \geq 0 \} \text{ is  $\gamma$-bounded} \big \}. 
\]
By convention we set $\omega^{\gamma}(A)=-\infty$ if $(T_t)_{t\geq 0}$ has no $\gamma$-type.
See Example \ref{nogtype} for simple examples of such semigroups.

When $X$ is a Hilbert space, the Gearhart-Prüss Theorem \cite[Theorem 5.2.1]{bat-ar} 
asserts that $\omega(A) = s_0(A)$. The main purpose of this section is to give 
an analogous equality $\omega^{\gamma}(A) = s_0^{\gamma}(A)$ on $K$-convex 
Banach spaces.

It is obvious that $\omega^{\gamma}(A) \leq \omega(A)$ and 
$s_0^{\gamma}(A) \leq s_0(A)$. The next inequality is more significant.

\begin{lem}\label{Ineq} Assume that $-A$ generates a $C_0$-semigroup. Then
$\omega(A) \leq s_0^{\gamma}(A)$.
\end{lem}

\begin{proof}
Let $\omega<\omega_0<\omega(A)$. By assumption, there exists a constant $M>0$ such that 
$\norme{T_s}\leq Me^{-\omega_0 s}$ for any $s\geq 0$.
Writing $e^{\omega s} T_s= e^{(\omega-\omega_0)s} e^{\omega_0 s} T_s$, we obtain that
$s\mapsto e^{\omega s} T_s x$ belongs to $L^1((0,\infty),X)$ for any $x\in X$, with
\begin{equation}\label{Ineq1}
\norme{s\mapsto e^{\omega s} T_s x}_{L^1((0,\infty),X)}\,\leq \frac{M}{\omega-\omega_0}\, \norme{x}.
\end{equation}
For any $\alpha \leq \omega$ and any $t\in\R$, we have 
\begin{align*}
R(\alpha + it,A)x &= - \int_{0}^{\infty} e^{its}e^{\alpha s}T(s)xds \\
& =- \int_{0}^{\infty} e^{its}e^{(\alpha-\omega)s}e^{\omega s}T_sxds.
\end{align*}
Since $\vert e^{its}e^{(\alpha-\omega)s}\vert\leq 1$ for any $s>0$,
we derive that the set 
\begin{equation}\label{Ineq2}
\{R(\alpha + it,A)\,:\, t \in \R,\, \alpha \leq \omega  \}
\end{equation}
is included in the set
\[
\Big\{ \displaystyle\int_{0}^{\infty}\psi(s)e^{\omega s}T_sds,\,:\, 
\psi\in L^{\infty}((0,\infty)), \;\norme{\psi}_{\infty} \leq 1 \Big\}.
\]
By Lemma \ref{Linftymeans} and (\ref{Ineq1}), the above set is
$\gamma$-bounded. Therefore the set (\ref{Ineq2})
is $\gamma$-bounded. Hence $\omega < s_0^{\gamma}(A)$. Passing to the 
supremum, this yields  $\omega(A) \leq s_0^{\gamma}(A)$.
\end{proof}

Summarizing, we have
\begin{equation} \label{ineGerPru}
\omega^{\gamma}(A)\leq \omega(A) \leq s_0^{\gamma}(A)  \leq s_0(A) 
\end{equation}
whenever $-A$ generates a $C_0$-semigroup.

\begin{thm}\label{GPth}
Let $X$ be a $K$-convex Banach space. Let $-A$ be the generator of a $C_0$-semigroup 
of $\gamma$-type $\omega$ on $X$. 
Then $A$ has a strong $\gamma$-1-bounded functional
calculus of type $s$ for each $s < s_0^{\gamma}(A)$.
\end{thm}

\begin{proof}
We fix some $s < s_0^{\gamma}(A)$. If $s\leq \omega$, then $A$ has a strong $\gamma$-1-bounded functional
calculus of type $s$ by Theorem \ref{shifengRthm}. Thus we may now assume that $\omega <s$.

Let $N\in \N$, let $x_1,\ldots,x_N \in X$ and let $y_1,\ldots, y_N \in X^{\odot}$. 
According to Theorem \ref{shifengRthm}, an estimate \eqref{ine1} is satisfied
for any $\alpha_1, \ldots, \alpha_N < \omega$.
Consider $s_1,...,s_N < s$ and let $\alpha_1,\ldots,\alpha_N < \omega$ be chosen such that 
\begin{equation}\label{Choice}
\omega-\alpha_k =s - s_k,\qquad k \in \N_N. 
\end{equation}
By the resolvent identity, we have
\[
R(s_k+it,A)x_k = \big(I+(\alpha_k-s_k)R(s_k +it,A)\big)R(\alpha_k+it,A)x_k
\]
for any $k \in \N_N$ and any $t\in \R$.
According to (\ref{Choice}), this implies that 
\[
\sqrt{s-s_k}R(s_k+it,A)x_k = \big(I+(\omega-s)R(s_k +it,A)\big)\sqrt{\omega-\alpha_k}R(\alpha_k+it,A)x_k.
\]
Now define $M_{s_1,...,s_n} : \R\times \N_N \rightarrow \mathcal{L}(X)$ by 
\[	
M_{s_1,...,s_n}(t,k) := I+(\omega-s)R(s_k+it,A),\quad t\in\R,\ k\in \N_N.
\]
The range of $M_{s_1,...,s_n}$ is included in the set 
$$
\{I + (\omega-s)R(\alpha+it,A),\; \alpha \leq s \},
$$ 
which is independent of $s_1,\cdots,s_k$. The latter set is $\gamma$-bounded, by the definition of $s_0^{\gamma}(A)$. 
Let $K>0$ denote its $\gamma$-bounded constant. 
Applying Theorem \ref{multithm} and \eqref{ine1}, we obtain that
\begin{align*}
\norme{(t,k)\mapsto \sqrt{s-s_k}R(s_k+it,A)x_k}_{\gamma(\R\times\N_N;X)} &
\leq K\norme{(t,k)\mapsto \sqrt{\omega-\alpha_k}R(\alpha_k+it,A)x_k}_{\gamma(\R\times\N_N;X)} \\
&\leq KC\norme{\sum_{k=1}^{N}{\gamma_k\otimes x_k}}_{G(X)}.
\end{align*}

Since $X$ is $K$-convex, the set $\{I + (\omega-s)R(\alpha+it,A)^*,\; \alpha \leq s \}$
is $\gamma$-bounded as well. Hence using \eqref{ine1*} we   obtain  a similar estimate
\begin{align*}
\norme{(t,k)\mapsto \sqrt{s-s_k}R(s_k+it,A^*)y_k}_{\gamma(\R\times\N_N;X^{*})}  
\leq K^*C\norme{\sum_{k=1}^{N}{\gamma_k\otimes y_k}}_{G(X^{*})}.
\end{align*}
Now applying the implication ``$(ii) \Rightarrow (iii)$'' of
Theorem \ref{shifengRthm}, we obtain the desired result. 
\end{proof}

\begin{cor}\label{corGP}
Let $X$ be a $K$-convex Banach space and let
$-A$ be the generator of a $C_0$-semigroup 
on $X$.  
If $\omega^{\gamma}(A) > -\infty$, then we have
$$
\omega^{\gamma}(A) = \omega(A) = s_0^{\gamma}(A).
$$ 
\end{cor}

\begin{proof}
If $\omega^{\gamma}(A) > -\infty$, then by
Theorem \ref{GPth}, $A$ has a strong $\gamma$-1-bounded functional calculus of type $s$ for
any $s < s_0^{\gamma}(A)$. According to Theorem \ref{shifengRthm}, this implies that 
$(T_t)_{t\geq 0}$ is of $\gamma$-type $s$ for
any $s < s_0^{\gamma}(A)$.
Thus $\omega^{\gamma}(A) \geq s_0^{\gamma}(A)$. Combining with (\ref{ineGerPru}), we obtain the result.
\end{proof}

\begin{ex}
Let $(\Omega,\mu)$ be a measure space and let $1<p<\infty$.
According to \cite[Theorem 5.3.6]{bat-ar}, if $-A$ is the generator of a 
\textit{positive} $C_0$-semigroup $(T_t)_{t\geq 0}$ on $L^p(\Omega)$, then $s_0(A) = \omega(A)$. 
If in addition $\omega^{\gamma}(A)>-\infty$ then the equalities 
$\omega^{\gamma}(A) = \omega(A) = s_0^{\gamma}(A) = s_0(A)$ hold by Corollary \ref{corGP}.
\end{ex}

The equality $\omega^{\gamma}(A) = \omega(A)$ in Corollary \ref{corGP} 
implies the following statement.

\begin{cor}\label{corGP2}
Let $(T_t)_{t\geq0}$ is a bounded $C_0$-semigroup on some $K$-convex Banach space.
If there exists $\delta >0$ such that $\{ e^{-\delta t} T_t\, :\, t\geq 0 \}$ is $\gamma$-bounded, 
then $\{ e^{-\delta t} T_t \, :\, t\geq 0 \}$ is $\gamma$-bounded for any $\delta > 0$.
\end{cor}

\begin{rq1}\label{rq_gambdd_genbdd}
Let $(T_t)_{t\geq0}$ is a bounded $C_0$-semigroup on $X$, with generator $-A$. If $A$ is bounded
(equivalently, if $(T_t)_{t\geq0}$  is uniformly continuous), then the property considered
in the above statement is true, that is,
$\{ e^{-\delta t} T_t, \; t\geq 0 \}$ is $\gamma$-bounded for any $\delta>0$.
	
Indeed, consider $\delta > 0$. Since $A$ is bounded and $\sigma(A) \subset \overline{R_0}$, 
there exists an open disk $D$ such that 
$\sigma(A + \delta ) \subset D \subset R_0$. Let $\partial D$ be the boundary of $D$
oriented counterclockwise. Then by the  Dunford-Riesz calculus, we have
\[
e^{-\delta  t }T_t = \frac{1}{2\pi i}\int_{\partial D}e^{-t\lambda}R(\lambda,A + \delta )d\lambda 
\]
for any $t\geq 0$. Then a straightforward application of Lemma \ref{Linftymeans}
shows that $\{e^{-\delta  t}T_t\, :\,  t \geq 0 \}$ is $\gamma$-bounded.	  
\end{rq1}

We conclude this section with an observation and two questions.
First we state a result that we recently obtained (with C. Le Merdy).

\begin{thm}\label{thm-ex2} (\cite[Corollary 0.5]{arn-lem})
Let $X$ be isomorphic to a separable Banach lattice with finite cotype such that 
$X$ is not isomorphic to an Hilbert space.
Then there exists $A \in \mathcal{L}(X)$ such that $\{e^{-tA}\, :\,t \geq 0\}$ is bounded but not $\gamma$-bounded.  
\end{thm}

Combining this theorem with Remark \ref{rq_gambdd_genbdd}, we obtain that Corollary \ref{corGP2} is sharp
in the class of uniformly continuous semigroups. Namely on 
any $K$-convex separable Banach lattice not isomorphic to a Hilbert space 
(on $L^p$ for $1<p\not=2<\infty$, say) we obtain a uniformly continuous semigroup
$(T_t)_{t\geq 0}$ such that 
$\{e^{-\delta  t}T_t\, :\,  t \geq 0 \}$ is $\gamma$-bounded for any 
$\delta>0$, $\{T_t\, :\,  t \geq 0 \}$ is bounded but 
$\{T_t\, :\,  t \geq 0 \}$ is not $\gamma$-bounded.

The assumption that $X$ is $K$-convex space in Corollary \ref{corGP2} is quite surprising. 
This leads to the following question:
\begin{q}
Does the assumption $X$ is $K$-convex space in Corollary  \ref{corGP2} can be dropped?  
\end{q}

We recall (\ref{ineGerPru}) and the existence of $-A$ generating a 
$C_0$-semigroup $(T_t)_{t\geq0}$ such that 
$\omega(A) < s_0(A) $. So we ask

\begin{q}
Does there exist an operator $A$ such that 
$-A$ which generates a $C_0$-semigroup $(T_t)_{t\geq0}$, 
satisfying $\omega(A) < s_0^{\gamma}(A)$ ?
\end{q}


\section{AN OVERVIEW}

Let $\omega\in\R$ and let $A$ be a half-plane type operator on some Banach space $X$.
Either in \cite{bat-haa} or in the present paper, the following six
properties are considered:

\vspace{0.5cm}
\begin{enumerate}[label = (\roman*),leftmargin=3.5cm ,parsep=0cm,itemsep=0cm,topsep=0cm]
\item $A$ has a bounded $H^{\infty}$-functional calculus of type $\omega$.
\item $A$ has a strong $1$-bounded functional calculus of type $\omega$.
\item $-A$ generates a $C_0$-semigroup of type $\omega$.
\end{enumerate} 
\vspace{0.3cm}
\begin{enumerate}[label = (\roman*$)_{\gamma}$,leftmargin=3.5cm ,parsep=0cm,itemsep=0cm,topsep=0cm]
\item $A$ has a $\gamma$-bounded $H^{\infty}$-functional calculus of type $\omega$.
\item $A$ has a strong $\gamma$-$1$-bounded functional calculus of type $\omega$.
\item $-A$ generates a $C_0$-semigroup of $\gamma$-type $\omega$.
\end{enumerate}
\vspace{0.5cm}

The aim of this last section is to give an overview of the relations between these properties,
at least on $K$-convex spaces.
This will require the analysis of a specific example, see Proposition \ref{Propequinotalpha}
below. In the above list, we have deliberately omitted
the strong $m$-bounded and $\gamma$-$m$-bounded functional calculi.

It follows from Proposition \ref{1implies2} and \cite[Theorem 6.4]{bat-haa} that
$$
(i) \Rightarrow (ii)\Rightarrow (iii).
$$

Likewise it follows from Remark \ref{rq1} (3) and Remark \ref{Add} that 
$$
(i)_{\gamma} \Rightarrow (ii)_{\gamma} \Rightarrow (iii)_{\gamma}.
$$

The implication ``$(ii) \Rightarrow (i)$'' is wrong. Indeed it follows from either
\cite{bai-cle} or \cite{mci-yag} that on any infinite dimensional Hilbert space
$H$, there exists a bounded $C_0$-semigroup $(T_t)_{t\geq 0}$ on $H$ whose 
negative generator $A$ does not have a bounded $H^{\infty}$-functional calculus of type $0$.
Thus with $\omega=0$, $A$ satisfies (iii) and does not satisfy (i). Moreover 
(ii) and (iii) are equivalent on Hilbert space, by \cite[Theorem 7.1]{bat-haa}.
This proves the result.

Since $\gamma$-boundedness and uniform boundedness are equivalent on Hilbert space, 
the above also shows that the implication ``$(ii)_\gamma \Rightarrow (i)_\gamma$'' is wrong.

The implication ``$(iii) \Rightarrow (ii)$'' is wrong. Indeed let $1 < p\neq 2 < +\infty$ and
let $(T_t)_{t \in \R}$ be the right translation group on $L^p(\R)$, which is a bounded $C_0$-group.
Let $-A$ denote its generator. It follows from Gomilko's paper \cite{gom} that either
$A$ or $-A$ does not have a strong $1$-bounded functional calculus of type $0$.

We have shown in Theorem \ref{shifengRthm}
that if $X$ is $K$-convex, then the implication ``$(iii)_\gamma \Rightarrow (ii)_\gamma$''
holds true. We do not know whether ``$(iii)_\gamma \Rightarrow (ii)_\gamma$'' holds true on any Banach space.

The implication ``$(iii)_{\gamma} \Rightarrow (ii)$'' holds true, by \cite[Theorem 6.4]{haa-roz} (see
Remark \ref{HR} for more on this.)

We noticed above that $(iii)$ does not imply $(i)$ on Hilbert space. Consequently, 
The implication ``$(iii)_\gamma \Rightarrow (i)$'' is wrong.

The only remaining question is whether $(i)$ implies $(iii)_\gamma$. We are going to show that 
this is wrong on sufficently bad spaces, see Example \ref{ex} below.

For this purpose we introduce a class of $C_0$-(semi)groups of independent interest. 
Recall the Gaussian space $G(X)$ from Subsection 2.3. We will use the 
so-called `contraction principle' \cite[Theorem 6.1.13]{hnvw}, which says 
that for any $x_1, \ldots , x_n \in X$ and any $\alpha_1, \ldots, \alpha_n \in \C$, we have
\begin{equation}\label{CP}
\norme{\sum_{k=1}^n \alpha_k \gamma_k\otimes x_k}_{G(X)} \leq 
\underset{k}\sup|\alpha_k| \norme{\sum_{k=1}^n \gamma_k\otimes x_k}_{G(X)}.
\end{equation}

We recall that $X$ has property $(\alpha)$ (see \cite[section 7.5]{hnvw} for more details) 
if there exists a constant $C\geq1$ such that for any finite family 
$(x_{ij})$ in $X$ and any finite family $(t_{ij})$ in $\C$, we have
\begin{equation}\label{propalpha}
\norme{\sum_{i,j}\gamma_i\otimes\gamma_j\otimes t_{ij}x_{ij}}_{G(G(X))} 
\leq C \underset{i,j}\sup|t_{ij}|\norme{\sum_{i,j}\gamma_i\otimes\gamma_j\otimes x_{ij}}_{G(G(X))}.
\end{equation}
Banach spaces with property $(\alpha)$ have a finite cotype, thus 
$\mathcal{R}$-boundedness and $\gamma$-boundedness are equivalent on such spaces. 
We recall that the class of all Banach spaces with property $(\alpha)$ 
is stable under taking subspaces and that
all Banach lattices with a finite cotype have property $(\alpha)$.
In particular, for any $1\leq p < \infty$, $L^p$-spaces 
and their subspaces have property $(\alpha)$.

Let $(\xi_k)_{k\geq 1}$ be a sequence of distinct points of $\R$. 
For any finite Gaussian sum $\sum_{k=1}^{n}\gamma_k\otimes x_k$, with $x_1,\ldots, x_n\in X$, we let 
\begin{equation}\label{DefSG}
T_t\Bigl(\sum_{k=1}^{n}\gamma_k\otimes x_k\Bigr) := \sum_{k=1}^n e^{-it\xi_k}\gamma_k\otimes x_k, \quad t\in\R.
\end{equation}
Then we have
\[
\norme{T_t\Bigl(\sum_{k=1}^{n}\gamma_k\otimes x_k\Bigr)} \leq  \norme{\sum_{k=1}^n \gamma_k\otimes x_k},
\]
by (\ref{CP}). 
Since the finite Gaussian sums are dense in $G(X)$, each $T_t$ extends to a bounded
linear operator on $G(X)$ (still denoted by $T_t$), with $\norme{T_t} = 1$.
Furthermore  $(T_t)_{t\in\R}$ is a $C_0$-group. 
Indeed it is plain that for any finite Gaussian sum $z=
\sum_{k=1}^{n}\gamma_k\otimes x_k$, $T_t(z)\to z$ when $t\to 0$.
Then the strong continuity of $(T_t)_{t\in\R}$ follows from the 
uniform boundedness of $(T_t)_{t\in\R}$ and the density of the set of all 
finite Gaussian sums in $G(X)$.

\begin{prop}\label{Propequinotalpha}
Let $-A$ denote the generator of the $C_0$-group $(T_t)_{t\in\R}$ defined by (\ref{DefSG}).
\begin{itemize}
\item [(1)] 
$A$ has a bounded $H^{\infty}$-functional calculus of type $0$.
\item [(2)] 
The $C_0$-group $(T_t)_{t \in \R}$ is $\gamma$-bounded if and only if $X$ has property $(\alpha)$.  
\end{itemize}		
\end{prop}

\begin{proof}
For any $b\in L^1(\R)$, we let $\int_{\R}b(t)T_tdt\,\in \mathcal{L}(G(X))$ denote the operator defined by
$$
\bigl(\int_{\R}b(t)T_tdt\bigr)(z)\, =\,\int_{0}^{\infty}b(t)T_t(z)dt,\quad z\in G(X).
$$
If $b \in L^1(\R_{+})$, we let $\mathcal{L}b$ denote the Laplace transform of $b$, that is,
\[
\mathcal{L}b(z) = \int_{0}^{\infty}e^{-zt}b(t)dt, \; z\in \overline{R_0}.
\]
Obviously $\mathcal{L}b$ is continuous and bounded on $\overline{R_0}$ and
its restriction to $R_0$ belongs to $H^{\infty}(R_0)$.

For any $b\in L^1(\R)$ and any $x_1,\ldots,x_n\in X$, we have
\begin{equation}\label{wfourb} 
\big(\int_{\R}b(t)T_tdt\big)
(\sum_{k=1}^n\gamma_k \otimes x_k) =	
\int_{\R}b(t)\sum_{k=1}^ne^{-i\xi_k t}\gamma_k \otimes x_k dt = 
\sum_{k=1}^n\widehat{b}(\xi_k)\gamma_k \otimes x_k.
\end{equation}
If $b \in L^1(\R_{+})$, this implies, using
(\ref{CP}),  that
\begin{equation}\label{Ex1-1}
\norme{\int_{0}^{\infty}b(t)T_tdt}_{\mathcal{L}(G(X))} 
\leq	 \norme{\mathcal{L}b}_{H^{\infty}(R_0)}.
\end{equation}

Now let $\alpha < 0$ and let $f \in \mathcal{E}(R_{\alpha})$. According to 
\cite[Lemma 5.1]{haa2}, there is a (necessarily unique) 
$b \in L^1(\R_+)$ such that $f = \mathcal{L}b$ on $R_0$ and
$$
f(A) = \int_{0}^{\infty}b(t)T_t dt.
$$
The estimate (\ref{Ex1-1}) therefore implies that 
\[
\norme{f(A)} \leq \norme{\mathcal{L}b}_{H^{\infty}(R_0)} = \norme{f}_{H^{\infty}(R_0)}.
\]
This shows (1).

We now turn to the proof of (2). First assume that 
$X$ has property $(\alpha)$. Let $(t_j)_j$ be a finite family of real numbers and 
for any $j$, let $z_j=\sum_k\gamma_j\otimes x_{jk}$ be a finite Gaussian sum.
We have
$$
\sum_j\gamma_j\otimes T_{t_j}(z_j)\, =\,\sum_{i,j} e^{-i t_j\xi_k}
\gamma_j\otimes\gamma_k\otimes x_{jk}.
$$
Applying \eqref{propalpha} we deduce that 
$$
\norme{\sum_j\gamma_j\otimes T_{t_j}(z_j)}_{G(G(X))}\,\leq \, C
\norme{\sum_{i,j}
\gamma_j\otimes\gamma_k\otimes x_{jk}}_{G(G(X))}\, =\, C\norme{\sum_j\gamma_j\otimes z_j}_{G(G(X))}.
$$
Since the set of all finite Gaussian sums is dense in $G(X)$, this shows that
$(T_t)_{t \in \R}$ is $\gamma$-bounded.

Assume on the contrary that $X$ does not have property $(\alpha)$.\hspace*{0.2cm}By (\ref{CP}), there exists a 
(necessarily unique) contractive,
non degenerate, homomorphism 
$$
w : C_0(\R) \rightarrow \mathcal{L}(G(X))
$$
such that
\[
w(f)(\sum_{k=1}^{n} \gamma_k \otimes x_k) = \sum_{k=1}^{n} f(\xi_k) \gamma_k\otimes x_k
\]
for any $n\geq 1$ and any $x_1,\ldots,x_n\in X$. We claim that
the set 
$$
S= \bigl\{w(f)\, :\, f \in C_0(\R),\ \norme{f}_\infty\leq 1 \bigr\}
$$ 
is not $\gamma$-bounded. 
Indeed let $(t_{ik})$ and $(x_{ik})$ be finite families in
$\C$ and $X$, respectively, and assume that $\vert t_{ik}\vert\leq 1$
for any $i,k$. There exist $f_i\in C_0(\R)$ such that $f_i(\xi_k)=t_{ik}$ 
and $\norme{f_i}_\infty\leq 1$ for any $i,k$. Then
$$
\sum_{i,k}\gamma_i\otimes \gamma_k\otimes t_{ik} x_{ik}
\, =\, \sum_i \gamma_i \otimes w(f_i)\Bigl(\sum_k \gamma_k\otimes x_{ik}\Bigr).
$$
If 
$S$ were $\gamma$-bounded, this would imply that the norm of the left hand side 
is dominated by the norm of $\sum_{i,k}\gamma_i\otimes \gamma_k\otimes x_{ik}$,
which would imply property $(\alpha)$.

Note that the homomorphismm $w$ `extends' $(T_t)_{t\in\R}$ in the sense 
of \cite[Definition 2.4]{lem1}. Indeed, according to \eqref{wfourb}, we have 
\[
\int_{\R}b(t)T_tdt= w(\hat b)
\]
for any $b\in L^1(\R)$. Therefore 
if $(T_t)_{t \in \R}$ were $\gamma$-bounded, then according to \cite[Theorem 4.4]{lem1}, 
the above set $S$ would be $\gamma$-bounded. We just noticed 
that this does not hold true. Hence, 
$(T_t)_{t \in \R}$ is not $\gamma$-bounded.
\end{proof}

\begin{ex}\label{ex}
Proposition
\ref{Propequinotalpha} provides an example of an operator which satisfies $(i)$ without satisfying 
$(iii)_{\gamma}$, for $\omega=0$. 
Indeed, assume that $X$ does not have property $(\alpha)$
and let $(T_t)_{t\in \R}$ (with generator $-A$)
be given by the above proposition. Changing $T_t$ into $T_{-t}$, part (1) of Proposition
\ref{Propequinotalpha} shows that $A$ and $-A$ have a bounded
$H^{\infty}$-bounded functional calculus of type $0$. However by part (2) of Proposition
\ref{Propequinotalpha}, either $(T_t)_{t \geq 0}$ or $(T_{-t})_{t \geq 0}$ is not $\gamma$-bounded. 
\end{ex}

We do not know if $(i)$ implies $(iii)_{\gamma}$ on non Hilbertian Banach spaces with property $(\alpha)$.
In particular, we do not know if $(i)$ implies $(iii)_{\gamma}$ on $L^p$-spaces, for $1<p\not= 2<\infty$.

\bibliographystyle{plain}
\bibliography{article}

\end{document}